\documentclass[11pt]{amsart}
\usepackage{amsmath,amssymb,amsthm,mathtools}
\usepackage[margin=1in]{geometry}
\usepackage{enumitem}
\usepackage{amsrefs}
\usepackage[hidelinks]{hyperref}

\numberwithin{equation}{section}

\theoremstyle{plain}
  \newtheorem{thm}{Theorem}[section]
  \newtheorem*{mthm}{Main Theorem}
  \newtheorem{lem}[thm]{Lemma}

  \newtheorem{prop}[thm]{Proposition}

\theoremstyle{definition}
  \newtheorem{defn}[thm]{Definition}

  \newtheorem{nota}[thm]{Notation}

\theoremstyle{remark}

  \newtheorem*{ack}{Acknowledgment}

\newcommand{\R}{\mathbb R}
\newcommand{\cL}{\mathcal L}
\newcommand{\ObsDiam}{\operatorname{ObsDiam}}
\newcommand{\PartDiam}{\operatorname{PartDiam}}
\newcommand{\Var}{\operatorname{Var}}
\newcommand{\Lip}{\operatorname{Lip}}
\newcommand{\diam}{\operatorname{diam}}

\title[Observable diameter of product measures]{Observable diameter of product measures\\
with a power-type interior zero}
\author{Takashi Shioya}
\address{Mathematical Institute, Tohoku University, Sendai 980-8578, Japan}
\email{shioya@math.tohoku.ac.jp}
\date{}
\thanks{This work was supported by JSPS KAKENHI Grant Numbers 24K06729.}

\makeatletter
\@namedef{subjclassname@2020}{
\textup{2020} Mathematics Subject Classification}
\makeatother

\begin{document}

\begin{abstract}
We determine the asymptotic order of the observable diameter of
$\ell^p$-product spaces associated with a probability measure on $[0,1]$
whose density has a power-type interior zero. Different vanishing orders
are allowed on the two sides of the zero. For every $1\le p\le\infty$,
we obtain sharp two-sided estimates and show that the asymptotic order is
governed by the larger of the two vanishing orders. In the case $p=2$,
the transition occurs exactly at the threshold for the spectral gap.
\end{abstract}

\maketitle

\section{Introduction}

The concentration of measure phenomenon was first discovered by P.~L\'evy
\cite{Levy} in his study of high-dimensional spheres and was brought to
prominence by V.~Milman \cite{MilmanDvoretzky} in the early 1970s through its
applications to asymptotic geometric analysis; see also \cite{Ledoux} and
\cite{ShioyaBook} for historical accounts and general background.
Within the theory of metric measure spaces (mm-spaces, for short), Gromov \cite{Gromov}
introduced the observable diameter as a numerical invariant of concentration of
measure, measuring how large an mm-space can appear through real-valued
$1$-Lipschitz functions. He also studied product spaces from this viewpoint.
In particular, his book contains an
estimate of the observable diameter of $\ell^1$-product spaces based on the
Laplace--Gibbs functional \cite{Gromov}*{\S~$3\frac12$.62}.
Gromov also gave a rough sketch of the corresponding upper estimate for
$p>1$, and Ozawa and Shioya \cite{OzawaShioyaProduct} later gave a complete
proof for all $1\le p<\infty$.  The resulting general upper bound has order
$n^{1/(2p)}$, and for $p>1$ a matching lower bound holds under a uniform
discreteness assumption.  Combined with the phase-transition framework
developed in \cite{OzawaShioyaLimit}, this gives a critical scale of
order $n^{-1/(2p)}$ in the uniformly discrete case.  For continuous spaces,
however, the corresponding lower bound need not hold, and the observable
diameter can have a substantially smaller order.  \cite{ShioyaProblems}*{Problem~9, p.~302}
asks for an estimate of
the observable diameter of the $n$-fold product of the unit interval, equipped
with the Euclidean metric and the product measure induced by a probability
density $\rho$ on the interval. It suggests that the behavior of $\rho$ near
its zeros should affect the asymptotic order. The
present paper gives a solution to that problem for densities with a power-type
interior zero, while allowing different powers on the two sides of the zero and,
more generally, $\ell^p$-product metrics.

Spectral gaps provide a second viewpoint.  The Poincar\'e inequality
tensorizes, and the Gromov--Milman concentration theorem
\cite{GromovMilman} then yields
dimension-free concentration for $\ell^2$-product powers.  For a compact
connected Riemannian manifold $X$, the identity
$\lambda_1(X^n)=\lambda_1(X)$ therefore gives a dimension-independent
upper bound for the observable diameter of $X^n$.  This classical
consequence is recorded explicitly in \cite{OzawaShioyaProduct}.

There is also a converse principle.  Gozlan, Roberto, and Samson \cite{GRS}
proved that, for an mm-space, a
nontrivial dimension-free concentration inequality for its $\ell^2$-product
powers is equivalent to a Poincar\'e inequality.  Thus the spectral gap is not
merely one sufficient source of dimension-free product concentration: in this
sense it characterizes the phenomenon.  Since dimension-free concentration
controls observable diameter at every fixed level, this viewpoint is especially
relevant for the present model.  The vanishing of the density creates a
one-dimensional bottleneck, and, for $p=2$, the threshold at which the
Poincar\'e inequality fails coincides with the threshold at which the observable
diameter changes its asymptotic order.

We now state the main theorem of this paper.  Let $\rho:[0,1]\to[0,\infty)$ be measurable,
assume that $\int_0^1\rho(x)\,dx=1$, and put $d\mu:=\rho\,d\cL^1$.
Assume that there exist $a\in(0,1)$, $\delta>0$, exponents
$\alpha_-,\alpha_+>0$, and constants $0<c_-\le C_-<\infty$ and
$0<c_+\le C_+<\infty$ such that
\begin{equation}\label{eq:power}
 \begin{aligned}
 c_-(a-x)^{\alpha_-}&\le \rho(x)\le C_-(a-x)^{\alpha_-}
 &&\text{for $\cL^1$-a.e.\ }x\in[0,1]\cap(a-2\delta,a),\\
 c_+(x-a)^{\alpha_+}&\le \rho(x)\le C_+(x-a)^{\alpha_+}
 &&\text{for $\cL^1$-a.e.\ }x\in[0,1]\cap(a,a+2\delta).
 \end{aligned}
\end{equation}
Assume moreover that there exists $c_\delta>0$ such that
\begin{equation}\label{eq:away-lower}
 \rho(x)\ge c_\delta
 \quad\text{for $\cL^1$-a.e.\ }x\in[0,1]\text{ with }|x-a|\ge\delta.
\end{equation}
Changing $\rho$ on the null set $\{a\}$ if necessary, we always take the
representative satisfying $\rho(a)=0$.  Put
\[
 \alpha_*:=\max\{\alpha_-,\alpha_+\}.
\]
For $1\le p\le\infty$, put
\[
 X_{n,p}:=([0,1]^n,d_p,\mu^{\otimes n}),
\]
where
\[
 d_p(x,y):=\left(\sum_{i=1}^n|x_i-y_i|^p\right)^{1/p}
 \quad(1\le p<\infty),
 \qquad
 d_\infty(x,y):=\max_{1\le i\le n}|x_i-y_i|.
\]
For positive sequences $A_n$ and $B_n$, the notation
$A_n\asymp_{\rho,p,\kappa}B_n$ means that there exist constants
$c,C>0$, depending only on $\rho$, $p$, and $\kappa$, such that
$cB_n\le A_n\le CB_n$ for all sufficiently large $n$.

\begin{mthm}
Fix $0<\kappa<1$. Then, as $n\to\infty$,
\[
 \ObsDiam(X_{n,p};-\kappa)\asymp_{\rho,p,\kappa}
 \begin{cases}
 n^{1/2},&p=1,\\[1mm]
 n^{1/p-1/2},&1<p\le2,\ \alpha_*<1/(p-1),\\[1mm]
 n^{1/p-1/2}(\log n)^{1/p},&1<p\le2,\ \alpha_*=1/(p-1),\\[1mm]
 n^{\frac1{2p}-\frac1{2(\alpha_*+1)}},&1<p\le2,\ \alpha_*>1/(p-1),\\[1mm]
 1,&2<p<\infty,\ \alpha_*\le p-1,\\[1mm]
 n^{\frac1{2p}-\frac1{2(\alpha_*+1)}},&2<p<\infty,\ \alpha_*>p-1,\\[1mm]
 1,&p=\infty.
 \end{cases}
\]
\end{mthm}

The theorem shows that, although the density may vanish with different powers
on the two sides of $a$, the asymptotic order is governed only by the larger
exponent $\alpha_*$.  For $1<p\le2$, the quantity
\[
 \int_0^1\rho(x)^{-(p-1)}\,dx
\]
is finite exactly when both $\alpha_-$ and $\alpha_+$ are smaller than
$1/(p-1)$, equivalently when $\alpha_*<1/(p-1)$.  If
$\alpha_*=1/(p-1)$, at least one side has logarithmic divergence and neither
side has polynomial divergence.  If $\alpha_*>1/(p-1)$, the polynomial
divergence on the side with the larger exponent dominates.  These three cases
are reflected in Main Theorem.

In particular, when $p=2$, the critical value is $\alpha_*=1$.  By the
one-dimensional Muckenhoupt criterion, $\mu$ has a positive spectral gap
exactly when $\alpha_-<1$ and $\alpha_+<1$, or equivalently when
$\alpha_*<1$.  Consequently
\[
 \ObsDiam(X_{n,2};-\kappa)\asymp
 \begin{cases}
 1,&\alpha_*<1,\\
 (\log n)^{1/2},&\alpha_*=1,\\
 n^{\frac14-\frac1{2(\alpha_*+1)}},&\alpha_*>1.
 \end{cases}
\]
Thus the loss of the spectral gap is again visible directly in the asymptotic
observable diameter.  For $2<p<\infty$, a different threshold,
$\alpha_*=p-1$, appears, so the spectral gap alone does not determine the
observable-diameter scale in that range.

We briefly indicate the proof strategy.  For $1<p\le2$, the upper bounds are
based on a weighted variance estimate.  We flatten the density independently
on the two sides of $a$, using separate regularization scales $r_-$ and $r_+$,
and compare the regularized and original product measures in Wasserstein
distance.  Choosing the two scales according to the corresponding exponents
$\alpha_-$ and $\alpha_+$ yields the sharp upper bounds.

For the lower bounds, the subcritical case follows from the $\ell^1$-product
lower bound of Ozawa and Shioya \cite{OzawaShioyaProduct}*{Theorem~1.3}, together
with the comparison $n^{1/p-1}d_1\le d_p$.  Above the critical exponent, a
majority-type observable is placed on a side where the vanishing order equals
$\alpha_*$.  At the critical exponent, a one-sided logarithmic ramp on a
critical side gives the required Hopf--Lax estimate, while Berry--Esseen
controls its fluctuations.

For $2<p<\infty$, we use a transport--entropy inequality of type $T_p$,
which controls the $p$-Wasserstein distance to the regularized measure in terms
of relative entropy.  We again regularize the two sides independently.  The
corresponding $T_p$ constant is governed by
\[
 \max\{1,r_-^{p-1-\alpha_-},r_+^{p-1-\alpha_+}\}.
\]
The transport costs within the two sides remain uniformly controlled; only the
cost of transferring mass through the bottleneck degenerates.  Tensorization,
Wasserstein comparison, and the two optimal regularization scales then give the
upper bound.

These arguments concern $1<p<\infty$.  The two endpoints are unchanged.  For
$p=\infty$, the metric diameter of $([0,1]^n,d_\infty)$ is $1$, while a
coordinate projection gives a positive lower bound independent of $n$.  For
$p=1$, the upper and lower bounds follow from
\cite{OzawaShioyaProduct}*{Theorems~1.1 and~1.3}, respectively.

For $1<p<\infty$, the resulting picture may again be viewed as a continuous
counterpart to the phase-transition framework for product mm-spaces developed
in \cite{OzawaShioyaLimit}: the scale is determined by the more singular side
of the interior bottleneck.

\section{Preliminaries}

\begin{defn}[mm-space]
An \emph{mm-space} is a triple $(X,d,\nu)$ consisting of a complete separable metric
space $(X,d)$ and a Borel probability measure $\nu$ on $X$.
\end{defn}
\begin{defn}[Partial diameter]
For an mm-space $(X,d,\nu)$ and $0<\kappa<1$, define the
\emph{partial diameter} by
\[
 \PartDiam(X,d,\nu;1-\kappa)
 :=\inf\{\diam B\mid B\subset X\text{ Borel},\ \nu(B)\ge1-\kappa\}.
\]
We also write $\PartDiam(\nu;1-\kappa)$ for
$\PartDiam(X,d,\nu;1-\kappa)$ when the underlying mm-space is clear.
\end{defn}

\begin{lem}\label{lem:positive-partdiam}
Let $\nu$ be a Borel probability measure on $[0,1]$ such that
$\nu\ll\cL^1$. Then, for every $0<\kappa<1$,
$\PartDiam(\nu;1-\kappa)>0$.
\end{lem}

\begin{proof}
By the absolute continuity of the integral, there exists $\varepsilon>0$ such
that $\cL^1(A)<\varepsilon$ implies $\nu(A)<1-\kappa$ for every Borel
$A\subset[0,1]$. If $B\subset\R$ is Borel and $\diam B<\varepsilon$, then
$B\cap[0,1]$ is contained in an interval of length less than $\varepsilon$,
and hence $\nu(B)=\nu(B\cap[0,1])<1-\kappa$. Therefore every Borel set
$B\subset\R$ with $\nu(B)\ge1-\kappa$ has $\diam B\ge\varepsilon$.
\end{proof}

\begin{defn}[Observable diameter]
For an mm-space $(X,d,\nu)$ and $0<\kappa<1$, define the
\emph{observable diameter} by
\[
 \ObsDiam(X,d,\nu;-\kappa)
 :=\sup\{\PartDiam(f_*\nu;1-\kappa)\mid f:X\to\R\text{ is $1$-Lipschitz}\}.
\]
See also \cite{ShioyaBook}*{Definition~2.13}.
\end{defn}

Note that $\ObsDiam(X,d,\nu;-\kappa)\le\PartDiam(X,d,\nu;1-\kappa)<\infty$
for every $0<\kappa<1$.

\begin{defn}[Couplings and Wasserstein distance]
For Borel probability measures $\sigma$ and $\tau$ on a metric space $(X,d)$,
let $\Pi(\sigma,\tau)$ denote the set of \emph{couplings} of $\sigma$ and
$\tau$, that is, the Borel probability measures $\pi$ on $X\times X$ whose
first and second marginals are $\sigma$ and $\tau$, respectively. For
$1\le p<\infty$, define the \emph{$p$-Wasserstein distance} by
\[
 W_p(\sigma,\tau)
 :=\left(
 \inf_{\pi\in\Pi(\sigma,\tau)}
 \int_{X\times X} d(x,y)^p\,d\pi(x,y)
 \right)^{1/p},
\]
with the convention that the infimum is $+\infty$ if every coupling has
infinite $p$-cost.
\end{defn}

\begin{nota}[Convention on constants] \label{nota}
Within a given proof, the positive constants $c$ and $C$ may change from line to
line, whereas named constants remain fixed after their definition within that
proof. For nonnegative quantities $A$ and $B$, we write $A\lesssim B$ if there
exists a constant $C>0$ such that $A\le CB$, and we write $A\gtrsim B$ if
$B\lesssim A$. We write $A\asymp B$ if both $A\lesssim B$ and $A\gtrsim B$
hold. Unless otherwise indicated, the implicit constants in $\lesssim$,
$\gtrsim$, and $\asymp$ may depend on $\rho$, $p$, $\kappa$, the fixed data
in \eqref{eq:power}--\eqref{eq:away-lower}, and fixed auxiliary parameters
appearing in the relevant statement or proof, but are uniform in all parameters
that are allowed to vary there. When $n$ is one of these parameters, it is
understood, unless otherwise indicated, to range over all sufficiently large
integers.
\end{nota}

\begin{defn}[Relative entropy and $T_p$ inequality]
For Borel probability measures $\nu$ and $\sigma$ on the same measurable
space $X$, define the \emph{relative entropy} by
\[
 H(\nu\mid\sigma)
 :=\begin{cases}
 \displaystyle\int_X \log\left(\frac{d\nu}{d\sigma}\right)\,d\nu,
 &\nu\ll\sigma,\\[2mm]
 +\infty,&\nu\not\ll\sigma.
 \end{cases}
\]
A probability measure $\sigma$ on a metric space is said to satisfy the
\emph{$T_p(C)$ inequality} if
$W_p(\nu,\sigma)^p\le C H(\nu\mid\sigma)$ for every Borel probability
measure $\nu$.
\end{defn}

\subsection{Berry--Esseen theorem}

Let $\Phi$ denote the distribution function of the standard Gaussian measure on $\R$ and $\mathbb E\xi$ the expectation of a random variable $\xi$.
See \cite{Petrov}*{Chapter~V, \S2, Theorem~4} for the proof of the following.

\begin{thm}[Berry--Esseen]\label{thm:BE}
There exists a universal constant $C_{\mathrm{BE}}$ such that, if
$\xi_1,\dots,\xi_n$ are independent and identically distributed real random
variables satisfying
\[
 \mathbb E\xi_1=0,
 \qquad
 \sigma^2:=\mathbb E\xi_1^2>0,
 \qquad
 \mathbb E|\xi_1|^3<\infty,
\]
then
\[
 \sup_{z\in\R}\left|
 \Pr\left(\frac{\sum_{i=1}^n\xi_i}{\sigma\sqrt n}\le z\right)-\Phi(z)
 \right|
 \le \frac{C_{\mathrm{BE}}\mathbb E|\xi_1|^3}{\sigma^3\sqrt n}.
\]
\end{thm}

\section{The range \texorpdfstring{$1<p\le2$}{1<p<=2}}

Fix $1 < p \le 2$ and put $q:=p/(p-1)\ge2$.  By decreasing $\delta$ if necessary and adjusting
$c_\delta$, we may assume from now on that
$[a-3\delta,a+3\delta]\subset(0,1)$.  Indeed, after replacing $\delta$ by
a smaller positive number $\delta'$, condition~\eqref{eq:away-lower} remains
valid with
\[
 c_{\delta'}:=\min\{c_\delta,c_-(\delta')^{\alpha_-},
 c_+(\delta')^{\alpha_+}\},
\]
in view of \eqref{eq:power}.

\subsection{Weighted one-dimensional estimate}

\begin{lem}\label{lem:weighted}
Let $d\nu:=w(x)\,dx$ be a probability measure on $[0,1]$ with $w>0$ $\cL^1$-a.e., and put
\[
 J(w):=\int_0^1w(x)^{-(p-1)}\,dx.
\]
Then every absolutely continuous $g:[0,1]\to\R$ satisfies
\[
 \Var_\nu(g)
 \le J(w)^{2/p}
 \left(\int_0^1|g'(x)|^q\,d\nu(x)\right)^{2/q},
\]
where $\Var_\nu(g)$ is the variance of $g$ with respect to $\nu$.
Consequently every $d_p$-$1$-Lipschitz $F:[0,1]^n\to\R$ satisfies
\[
 \Var_{\nu^{\otimes n}}(F)
 \le J(w)^{2/p}n^{1-2/q}.
\]
\end{lem}

\begin{proof}
Let $0\le x<y\le1$. Since $g$ is absolutely continuous, H\"older's inequality on $[x,y]$, with exponents $q$ and $p$, gives
\[
 \begin{aligned}
 |g(y)-g(x)|
 &\le\int_x^y|g'(t)|\,dt\\
 &\le
 \left(\int_x^y|g'(t)|^qw(t)\,dt\right)^{1/q}
 \left(\int_x^yw(t)^{-p/q}\,dt\right)^{1/p}.
 \end{aligned}
\]
Since $p/q=p-1$,
\[
 \int_x^yw(t)^{-p/q}\,dt
 \le\int_0^1w(t)^{-(p-1)}\,dt
 =J(w),
\]
and $\int_x^y|g'(t)|^qw(t)\,dt \le\int_0^1|g'(t)|^q\,d\nu(t)$. Therefore $|g(y)-g(x)| \le J(w)^{1/p} \left(\int_0^1|g'|^q\,d\nu\right)^{1/q}$,
which yields
\[
 \operatorname{osc}_{[0,1]}g
 := \sup_{x,y \in [0,1]} |g(y) - g(x)|
 \le J(w)^{1/p}
 \left(\int_0^1|g'|^q\,d\nu\right)^{1/q}.
\]
For every real-valued random variable $Z$ whose essential range is contained
in an interval of length $L$, one has $\Var(Z)\le L^2$; indeed, $\Var(Z)=\inf_{c\in\R}\mathbb E|Z-c|^2 \le\mathbb E|Z-t_0|^2\le L^2$ for any $t_0$ belonging to that interval. Applying this with $Z=g$ gives $\Var_\nu(g) \le J(w)^{2/p} \left(\int_0^1|g'|^q\,d\nu\right)^{2/q}$. This proves the one-dimensional assertion.

We now prove the product estimate. By the McShane extension theorem, $F$ extends
from $[0,1]^n$ to a $1$-Lipschitz function on $\R^n$ for the $\ell^p$ norm;
we denote the extension again by $F$. By Rademacher's theorem \cite{EvansGariepy}*{Section~3.1.2},
$\|\nabla F\|_q\le1$ at $\cL^n$-a.e.\ point.

Let $X_1,\ldots,X_n$ be independent with law $\nu$. The Efron--Stein
inequality \cite{EfronStein}*{Section~2} gives
\[
 \Var_{\nu^{\otimes n}}(F)
 \le \sum_{i=1}^n
 \int_{[0,1]^{n-1}}\Var_\nu
 \bigl(F(x_1,\ldots,x_{i-1},\cdot,x_{i+1},\ldots,x_n)\bigr)
 \,d\nu^{\otimes(n-1)}.
\]
For every fixed choice of the other coordinates, the corresponding
one-dimensional slice of $F$ is $1$-Lipschitz and hence absolutely
continuous. Let $N$ be the $\cL^n$-null set at whose points $F$ is not
differentiable. By Fubini's theorem, for $\cL^{n-1}$-a.e.\ choice of the
other coordinates, the corresponding section of $N$ has $\cL^1$-measure
zero. Since $\nu^{\otimes(n-1)}\ll\cL^{n-1}$ and $\nu\ll\cL^1$, for
$\nu^{\otimes(n-1)}$-a.e.\ choice of the other coordinates the derivative
of the slice agrees with $\partial_iF$ for $\nu$-a.e.\ value of the
$i$-th coordinate. The one-dimensional estimate therefore gives
\[
 \Var_\nu
 \bigl(F(x_1,\ldots,x_{i-1},\cdot,x_{i+1},\ldots,x_n)\bigr)
 \le J(w)^{2/p}
 \left(\int_0^1|\partial_iF|^q\,d\nu\right)^{2/q}.
\]
Since $2/q\le1$, Jensen's inequality therefore gives
\[
 \Var_{\nu^{\otimes n}}(F)
 \le J(w)^{2/p}
 \sum_{i=1}^n
 \left(\int_{[0,1]^n}|\partial_iF|^q\,d\nu^{\otimes n}\right)^{2/q}.
\]
Put $a_i:=\int_{[0,1]^n}|\partial_iF|^q\,d\nu^{\otimes n}$. Since
$q/2\ge1$, H\"older's inequality for the finite sum gives
\[
 \sum_{i=1}^na_i^{2/q}
 \le n^{1-2/q}\left(\sum_{i=1}^na_i\right)^{2/q}.
\]
Finally, $\nu^{\otimes n}\ll\cL^n$ and $\|\nabla F\|_q\le1$
$\cL^n$-a.e., so
\[
 \sum_{i=1}^na_i
 =\int_{[0,1]^n}\sum_{i=1}^n|\partial_iF|^q\,d\nu^{\otimes n}
 \le1.
\]
Combining these estimates proves
$\Var_{\nu^{\otimes n}}(F)\le J(w)^{2/p}n^{1-2/q}$.
\end{proof}

\subsection{Wasserstein comparison}

\begin{lem}\label{lem:Wp-comp}
Let $(X,d)$ be a metric space, let $\nu,\nu'$ be Borel probability measures, let $1\le p<\infty$, and let $0<\eta<\kappa<1$. Then
\[
 \ObsDiam(X,d,\nu;-\kappa)
 \le \ObsDiam(X,d,\nu';-(\kappa-\eta))+2\eta^{-1/p}W_p(\nu,\nu').
\]
\end{lem}

\begin{proof}
If $W_p(\nu,\nu')=+\infty$, there is nothing to prove. Set
$D:=W_p(\nu,\nu')<\infty$. Fix $\varepsilon>0$ and choose
$\pi\in\Pi(\nu,\nu')$ such that
$\int_{X\times X}d(x,y)^p\,d\pi(x,y)\le D^p+\varepsilon$. Put
$r:=\eta^{-1/p}(D^p+\varepsilon)^{1/p}$. Then Markov's inequality gives
\[
 \pi\bigl(\{(x,y)\in X\times X\mid d(x,y)>r\}\bigr)\le\eta.
\]
Let $f:X\to\R$ be $1$-Lipschitz and let $\varepsilon'>0$. By the definition
of partial diameter, there is a nonempty Borel set $B\subset\R$ such that
$(f_*\nu')(B)\ge1-\kappa+\eta$ and
\[
 \diam B\le \ObsDiam(X,d,\nu';-(\kappa-\eta))+\varepsilon'.
\]
Since $B$ has finite diameter, $I:=[\inf B,\sup B]$ satisfies
$(f_*\nu')(I)\ge1-\kappa+\eta$ and $\diam I=\diam B$. Let
$I^{[r]}:=\{t\in\R\mid \operatorname{dist}(t,I)\le r\}$, where
$\operatorname{dist}(t,I):=\inf_{s\in I}|t-s|$.
If $f(y)\in I$ and
$d(x,y)\le r$, then $f(x)\in I^{[r]}$. Hence, using the two marginals of
$\pi$ and the preceding Markov estimate,
\[
 \begin{aligned}
 (f_*\nu)(I^{[r]})
 &\ge \pi\bigl(\{f(y)\in I,\ d(x,y)\le r\}\bigr)\\
 &\ge (f_*\nu')(I)-\pi\bigl(\{d(x,y)>r\}\bigr)\ge1-\kappa.
 \end{aligned}
\]
Moreover $\diam I^{[r]}=\diam I+2r$. Therefore
\[
 \PartDiam(f_*\nu;1-\kappa)
 \le\ObsDiam(X,d,\nu';-(\kappa-\eta))+\varepsilon'+2r.
\]
Taking the supremum over all $1$-Lipschitz $f$ and then letting
$\varepsilon,\varepsilon'\downarrow0$ proves the assertion.
\end{proof}

\subsection{Two-sided regularization}

\begin{defn}[Two-sided regularization] \label{defn:two-reg}
For $0<r_-,r_+<\delta$, put $\mathbf r:=(r_-,r_+)$ and define
\[
 c_{r_-,-}:=\frac1{r_-}\int_0^{r_-}\rho(a-s)\,ds,
 \qquad
 c_{r_+,+}:=\frac1{r_+}\int_0^{r_+}\rho(a+s)\,ds,
\]
and
\[
 \rho_{\mathbf r}(x):=\begin{cases}
 \rho(x),&x\in[0,a-r_-)\cup(a+r_+,1],\\
 c_{r_-,-},&x\in[a-r_-,a],\\
 c_{r_+,+},&x\in(a,a+r_+].
 \end{cases}
\]
We call $\rho_{\mathbf r}$ the \emph{two-sided regularization} of $\rho$ at
the scales $r_-$ and $r_+$.
The two flattened pieces preserve the corresponding side masses:
\[
 \int_{a-r_-}^a\rho_{\mathbf r}\,d\cL^1=\int_{a-r_-}^a\rho\,d\cL^1,
 \qquad
 \int_a^{a+r_+}\rho_{\mathbf r}\,d\cL^1=\int_a^{a+r_+}\rho\,d\cL^1.
\]
Thus $d\mu_{\mathbf r}:=\rho_{\mathbf r}\,d\cL^1$ is a Borel probability
measure. Moreover,
$c_{r_-,-}\asymp r_-^{\alpha_-}$ and
$c_{r_+,+}\asymp r_+^{\alpha_+}$, with constants independent of
$r_-,r_+$. Hence
\begin{equation}\label{eq:rhor}
 \begin{aligned}
 \rho_{\mathbf r}(a-s)&\asymp
 \begin{cases}
 r_-^{\alpha_-},&0\le s\le r_-,\\
 s^{\alpha_-},&r_-<s<\delta,
 \end{cases}\\
 \rho_{\mathbf r}(a+s)&\asymp
 \begin{cases}
 r_+^{\alpha_+},&0\le s\le r_+,\\
 s^{\alpha_+},&r_+<s<\delta.
 \end{cases}
 \end{aligned}
\end{equation}
where the comparisons on the second branches are understood for
$\cL^1$-a.e.\ $s$.
\end{defn}

\begin{nota}
For $\alpha>0$ and $0<r<\delta$, put
\[
 \mathcal J_\alpha(r):=
 \begin{cases}
 1,&\alpha(p-1)<1,\\
 \log(1/r),&\alpha(p-1)=1,\\
 r^{1-\alpha(p-1)},&\alpha(p-1)>1.
 \end{cases}
\]
\end{nota}

\begin{lem}\label{lem:Jr}
For $1<p\le2$,
\[
 J(\rho_{\mathbf r})\asymp
 \mathcal J_{\alpha_-}(r_-)+\mathcal J_{\alpha_+}(r_+).
\]
For every $1\le p<\infty$, there exists $C_p>0$ such that
\begin{equation}\label{eq:Wp-reg}
 W_p(\mu,\mu_{\mathbf r})^p
 \le C_p\left(r_-^{\alpha_-+p+1}+r_+^{\alpha_++p+1}\right),
\end{equation}
and
\[
 W_p(\mu^{\otimes n},\mu_{\mathbf r}^{\otimes n})
 \le C_pn^{1/p}
 \left(r_-^{\alpha_-+p+1}+r_+^{\alpha_++p+1}\right)^{1/p}.
\]
\end{lem}

\begin{proof}
We first estimate $J(\rho_{\mathbf r})$. Put
$\gamma_\pm:=\alpha_\pm(p-1)$ and
$K_\delta:=[0,1]\setminus(a-\delta,a+\delta)$. By
\eqref{eq:away-lower}, the contribution from $K_\delta$ is bounded above
independently of $\mathbf r$. Since $\int_0^1\rho_{\mathbf r}(x)\,dx=1$,
Jensen's inequality also gives
$J(\rho_{\mathbf r})\ge1$.

Consider the left side. By \eqref{eq:rhor},
\[
 \int_{a-r_-}^a\rho_{\mathbf r}(x)^{-(p-1)}\,dx
 \asymp r_-^{1-\gamma_-},
 \qquad
 \int_{a-\delta}^{a-r_-}\rho_{\mathbf r}(x)^{-(p-1)}\,dx
 \asymp\int_{r_-}^{\delta}s^{-\gamma_-}\,ds.
\]
The last integral equals
\[
 \begin{cases}
 \dfrac{\delta^{1-\gamma_-}-r_-^{1-\gamma_-}}{1-\gamma_-},&\gamma_-<1,\\[2mm]
 \log(\delta/r_-),&\gamma_-=1,\\[2mm]
 \dfrac{r_-^{1-\gamma_-}-\delta^{1-\gamma_-}}{\gamma_--1},&\gamma_->1.
 \end{cases}
\]
Thus, if $\gamma_-<1$, both integrals above are $O(1)$, and hence so
is the left contribution. If $\gamma_-=1$, the left contribution is
comparable to $1+\log\frac{\delta}{r_-}\asymp\log\frac1{r_-}$.
Finally, if $\gamma_->1$, the left contribution is comparable to
$r_-^{1-\gamma_-}$. The same argument on the right gives
the corresponding alternatives for $\alpha_+,r_+$. Combining the two
side estimates with $J(\rho_{\mathbf r})\ge1$ proves the first assertion.

We now prove the Wasserstein estimate. By construction, $\mu$ and
$\mu_{\mathbf r}$ have the same total mass on each of $[a-r_-,a]$ and
$[a,a+r_+]$, and they agree outside these intervals. Couple the two
restrictions on each side and use the diagonal coupling on the complement.
Under the resulting coupling $\pi_{\mathbf r}$, every point moved on the
left travels at most $r_-$ and every point moved on the right travels at most
$r_+$. Hence
\[
 \int_{[0,1]^2}|x-y|^p\,d\pi_{\mathbf r}
 \le r_-^p\mu([a-r_-,a])+r_+^p\mu([a,a+r_+]).
\]
By \eqref{eq:power},
$\mu([a-r_-,a])\le C r_-^{\alpha_-+1}$ and
$\mu([a,a+r_+])\le C r_+^{\alpha_++1}$. This gives
\eqref{eq:Wp-reg}. The product measure
$\pi_{\mathbf r}^{\otimes n}$ is a coupling of $\mu^{\otimes n}$ and
$\mu_{\mathbf r}^{\otimes n}$. Since
$d_p(x,y)^p=\sum_{i=1}^n|x_i-y_i|^p$, Fubini's theorem gives
\begin{align*}
 W_p(\mu^{\otimes n},\mu_{\mathbf r}^{\otimes n})^p
 &\le \int_{[0,1]^{2n}}d_p(x,y)^p\,d\pi_{\mathbf r}^{\otimes n}(x,y)\\
 &=n\int_{[0,1]^2}|s-t|^p\,d\pi_{\mathbf r}(s,t)\\
 &\le Cn\left(r_-^{\alpha_-+p+1}+r_+^{\alpha_++p+1}\right).
\end{align*}
Taking the $p$-th root proves the second estimate.
\end{proof}

\subsection{Upper and baseline lower bounds}

\begin{prop}\label{prop:upper-le2}
For $1<p\le2$ and every fixed $0<\kappa<1$, as $n\to\infty$,
\[
 \ObsDiam(X_{n,p};-\kappa)
 \lesssim
 \begin{cases}
 n^{1/p-1/2},&\alpha_*<1/(p-1),\\
 n^{1/p-1/2}(\log n)^{1/p},&\alpha_*=1/(p-1),\\
 n^{\frac1{2p}-\frac1{2(\alpha_*+1)}},&\alpha_*>1/(p-1).
 \end{cases}
\]
\end{prop}

See Notation~\ref{nota} for $\lesssim$.

\begin{proof}
We treat the three cases separately.

Assume first that $\alpha_*<1/(p-1)$. Then
$\gamma_\pm:=\alpha_\pm(p-1)<1$. By \eqref{eq:power},
\[
 \int_{a-\delta}^a\rho(x)^{-(p-1)}\,dx
 \le c_-^{-(p-1)}\int_0^\delta s^{-\gamma_-}\,ds<\infty,
\]
and similarly
$\int_a^{a+\delta}\rho(x)^{-(p-1)}\,dx<\infty$. On
$K_\delta:=[0,a-\delta]\cup[a+\delta,1]$, \eqref{eq:away-lower}
gives a finite contribution. Thus $J(\rho)<\infty$.

Let $F:[0,1]^n\to\R$ be $d_p$-$1$-Lipschitz and put
\[
 m_F:=\int_{[0,1]^n} F\,d\mu^{\otimes n}.
\]
Lemma~\ref{lem:weighted} gives
\[
 \Var_{\mu^{\otimes n}}(F)\le J(\rho)^{2/p}n^{1-2/q}.
\]
Hence
Chebyshev's inequality yields
\[
 \mu^{\otimes n}(|F-m_F|>t)
 \le \frac{J(\rho)^{2/p}n^{1-2/q}}{t^2}.
\]
Choosing
\[
 t:=\kappa^{-1/2}J(\rho)^{1/p}n^{1/2-1/q},
\]
and using
$1/2-1/q=1/p-1/2$, we obtain an interval of mass at least $1-\kappa$ and
diameter at most $C_{\rho,p,\kappa}n^{1/p-1/2}$. Taking the supremum over
$F$ proves the subcritical upper bound.

Assume now $\alpha_*\ge1/(p-1)$. Fix $\eta\in(0,\kappa)$ and, for all
sufficiently large $n$, put
\[
 r_{\pm,n}:=n^{-1/(2(\alpha_\pm+1))},
 \qquad \mathbf r_n:=(r_{-,n},r_{+,n}).
\]
Let $F$ be $d_p$-$1$-Lipschitz and put
$m_{F,\mathbf r_n}:=\int_{[0,1]^n} F\,d\mu_{\mathbf r_n}^{\otimes n}$.
Lemma~\ref{lem:weighted} gives
\[
 \Var_{\mu_{\mathbf r_n}^{\otimes n}}(F)
 \le J(\rho_{\mathbf r_n})^{2/p}n^{1-2/q}.
\]
Hence, for every $t>0$, Chebyshev's inequality yields
\[
 \mu_{\mathbf r_n}^{\otimes n}
 \bigl(|F-m_{F,\mathbf r_n}|>t\bigr)
 \le \frac{J(\rho_{\mathbf r_n})^{2/p}n^{1-2/q}}{t^2}.
\]
Taking
\[
 t:=(\kappa-\eta)^{-1/2}
 J(\rho_{\mathbf r_n})^{1/p}n^{1/2-1/q}
\]
produces an interval centered at $m_{F,\mathbf r_n}$ of
$\mu_{\mathbf r_n}^{\otimes n}$-mass at least $1-(\kappa-\eta)$ and
diameter $2t$. Since $1/2-1/q=1/p-1/2$, taking the supremum over all
$d_p$-$1$-Lipschitz $F$ gives
\[
 \ObsDiam([0,1]^n,d_p,\mu_{\mathbf r_n}^{\otimes n};-(\kappa-\eta))
 \le C J(\rho_{\mathbf r_n})^{1/p}n^{1/p-1/2}.
\]
Lemma~\ref{lem:Wp-comp} and Lemma~\ref{lem:Jr} therefore give
\begin{equation}\label{eq:upper-reg-le2}
 \ObsDiam(X_{n,p};-\kappa)
 \le C\left[
 n^{1/p-1/2}J(\rho_{\mathbf r_n})^{1/p}
 +n^{1/p}
 \left(r_{-,n}^{\alpha_-+p+1}+r_{+,n}^{\alpha_++p+1}\right)^{1/p}
 \right].
\end{equation}
For $\omega\in\{-,+\}$ put
\[
 \beta_\omega:=\frac1{2p}-\frac1{2(\alpha_\omega+1)}.
\]
Then a direct calculation gives
\begin{equation}\label{eq:side-Wp-scale}
 n^{1/p}r_{\omega,n}^{(\alpha_\omega+p+1)/p}=n^{\beta_\omega}.
\end{equation}
The function $\alpha\mapsto1/(2p)-1/(2(\alpha+1))$ is increasing, and at
$\alpha=1/(p-1)$ its value is $1/p-1/2$.

Suppose $\alpha_*=1/(p-1)$. At least one side is critical and neither is
supercritical. Lemma~\ref{lem:Jr} gives
$J(\rho_{\mathbf r_n})\le C\log n$, while
$\beta_\pm\le1/p-1/2$. Hence the first term in
\eqref{eq:upper-reg-le2} is at most
$Cn^{1/p-1/2}(\log n)^{1/p}$, and by
\eqref{eq:side-Wp-scale} the second is at most $Cn^{1/p-1/2}$. This proves
the critical upper bound.

Finally suppose $\alpha_*>1/(p-1)$ and put
$\beta_*:=1/(2p)-1/(2(\alpha_*+1))$. Then
$\beta_*>1/p-1/2$. A subcritical side contributes $O(1)$ to
$J(\rho_{\mathbf r_n})$, a critical side contributes $O(\log n)$, and a
supercritical side with exponent $\alpha_\omega$ contributes
\[
 \mathcal J_{\alpha_\omega}(r_{\omega,n})^{1/p}
 =r_{\omega,n}^{(1-\alpha_\omega(p-1))/p}.
\]
Multiplying this last quantity by $n^{1/p-1/2}$ gives exactly
$n^{\beta_\omega}$, because
\[
 \frac1p-\frac12-
 \frac{1-\alpha_\omega(p-1)}{2p(\alpha_\omega+1)}
 =\frac1{2p}-\frac1{2(\alpha_\omega+1)}.
\]
Since $\beta_\omega\le\beta_*$, and since
$n^{1/p-1/2}(\log n)^{1/p}=o(n^{\beta_*})$, we obtain
$n^{1/p-1/2}J(\rho_{\mathbf r_n})^{1/p}\le Cn^{\beta_*}$. Likewise,
using \eqref{eq:side-Wp-scale} and
$(u+v)^{1/p}\le u^{1/p}+v^{1/p}$, the Wasserstein term in
\eqref{eq:upper-reg-le2} is at most
$C(n^{\beta_-}+n^{\beta_+})\le Cn^{\beta_*}$. This proves the
supercritical upper bound.
\end{proof}

We use the following.

\begin{thm}[{\rm [16, Theorem~1.3]}]\label{thm:OzawaShioya}
Let $(X,d,\nu)$ be a nontrivial mm-space with finite diameter, and
write
\[
 X_1^n:=(X^n,d_1,\nu^{\otimes n}),
 \qquad
 d_1(x,y):=\sum_{i=1}^n d(x_i,y_i),
\]
where $x = (x_1,\dots,x_n)$, $y = (y_1,\dots,y_n) \in X^n$.
Put
\[
 V(X,d,\nu) := \sup_{x\in X}\Var_\nu(d(x,\cdot)),
\]
and let $\varphi_1(t):=\sqrt{2/\pi}\,e^{-t^2/2}$ for $t\ge0$ and
$\varphi_1(t):=0$ for $t<0$. Then, for every $0<\kappa<1$,
\[
 \liminf_{n\to\infty}
 \frac{\ObsDiam(X_1^n;-\kappa)}{\sqrt n}
 \ge \sqrt{V(X,d,\nu)}\,\PartDiam(\varphi_1(t)\,dt;1-\kappa).
\]
\end{thm}

\begin{lem}\label{lem:baseline}
For $1\le p\le2$ and every fixed $0<\kappa<1$, there exists
$c_{p,\kappa}>0$ such that
\[
 \ObsDiam(X_{n,p};-\kappa)\ge c_{p,\kappa}n^{1/p-1/2}
 \qquad(n\ge1).
\]
For $p>2$ and every fixed $0<\kappa<1$, there exists $c_\kappa>0$ such that
\[
 \ObsDiam(X_{n,p};-\kappa)\ge c_\kappa
 \qquad(n\ge1).
\]
\end{lem}

\begin{proof}
The mm-space $([0,1],|\cdot|,\mu)$ is nontrivial and has finite diameter.
Applying the preceding theorem, its lower-bound constant is positive for this
space: $V([0,1],|\cdot|,\mu)>0$, and the partial diameter of the absolutely continuous limit
measure is positive at every level $1-\kappa>0$.
Therefore, for all sufficiently large $n$,
\[
 \ObsDiam(X_{n,1};-\kappa)\ge c\sqrt n.
\]
Let $1<p\le2$ and put $q:=p/(p-1)$. For $x,y\in[0,1]^n$, H\"older's
inequality gives
\[
 n^{1/p-1}d_1(x,y)
 =n^{1/p-1}\sum_{i=1}^n|x_i-y_i|
 \le d_p(x,y).
\]
Hence the identity map
\[
 ([0,1]^n,d_p,\mu^{\otimes n})
 \longrightarrow
 ([0,1]^n,n^{1/p-1}d_1,\mu^{\otimes n})
\]
is $1$-Lipschitz. By the definition of observable diameter and its homogeneity
under scaling of the metric,
\[
 \ObsDiam(X_{n,p};-\kappa)
 \ge n^{1/p-1}\ObsDiam(X_{n,1};-\kappa)
 \ge c n^{1/2-1/q}
 =c n^{1/p-1/2}
\]
for all sufficiently large $n$. The same conclusion for $p=1$ is exactly the
preceding $\ell^1$ estimate. For the remaining finitely many $n$, the
first-coordinate projection and Lemma~\ref{lem:positive-partdiam} give a
positive lower bound; decreasing $c$ if necessary proves the first assertion
for every $n\ge1$.

For $p>2$, the first-coordinate projection is $1$-Lipschitz and pushes
$\mu^{\otimes n}$ forward to $\mu$. Hence Lemma~\ref{lem:positive-partdiam}
gives
\[
 \ObsDiam(X_{n,p};-\kappa)\ge\PartDiam(\mu;1-\kappa)>0,
\]
which proves the second assertion.
\end{proof}

\subsection{Block amplification}

\begin{lem}\label{lem:block}
Let $(Y,d,\nu)$ be an mm-space. Fix $1<p<\infty$ and $0<\kappa_0<1$. Then for every $0<\kappa<1$ there exist $k=k(\kappa,\kappa_0)$ and $c=c(p,\kappa,\kappa_0)>0$ such that
\[
 \ObsDiam(Y^k,d_{p,k},\nu^{\otimes k};-\kappa)
 \ge c\,\ObsDiam(Y,d,\nu;-\kappa_0),
\]
where
\[
 d_{p,k}((x_j),(y_j)):=\left(\sum_{j=1}^k d(x_j,y_j)^p\right)^{1/p}.
\]
More precisely, writing $q:=p/(p-1)$, one may take $c=k^{-1/q}$.
\end{lem}

\begin{proof}
Put $D:=\ObsDiam(Y,d,\nu;-\kappa_0)$. If $D=0$, the assertion is immediate. Assume $D>0$ and
fix an arbitrary $L$ with $0<L<D$. By the definition of observable diameter,
there exists a $1$-Lipschitz map $f:Y\to\R$ such that $\PartDiam(f_*\nu;1-\kappa_0)>L$. Let $\lambda:=f_*\nu$ and set $\delta:=\kappa_0/4$. Let
\[
 F_\lambda(t):=\lambda(( -\infty,t]),
 \qquad
 F_\lambda^{-1}(u):=\inf\{t\in\R\mid F_\lambda(t)\ge u\}
 \quad(0<u<1),
\]
and define $a_0:=F_\lambda^{-1}(\delta)$, $b_0:=F_\lambda^{-1}(1-\delta)$. Since $\delta<1-\delta$ and $F_\lambda^{-1}$ is nondecreasing, $a_0\le b_0$. By the definition of the generalized inverse, $\lambda(( -\infty,a_0))\le\delta\le\lambda(( -\infty,a_0])$ and $\lambda(( -\infty,b_0))\le1-\delta\le\lambda(( -\infty,b_0])$. The second pair of inequalities is equivalent to $\lambda((b_0,\infty))\le\delta\le\lambda([b_0,\infty))$. Then
\[
 \lambda([a_0,b_0])
 \ge1-\lambda(( -\infty,a_0))-\lambda((b_0,\infty))
 \ge1-2\delta>1-\kappa_0.
\]
If $b_0-a_0<L$, the interval $[a_0,b_0]$ would have diameter strictly
smaller than $L$ and $\lambda$-mass greater than $1-\kappa_0$, contrary to $\PartDiam(\lambda;1-\kappa_0)>L$. Hence $b_0-a_0\ge L$.

Define $Y_0(y):=\min\{\max\{f(y)-a_0,0\},L\}$, $y\in Y$. The truncation map $t\mapsto\min\{\max\{t-a_0,0\},L\}$ is $1$-Lipschitz on $\R$, hence $Y_0$ is $1$-Lipschitz. Since $\{f\le a_0\}\subset\{Y_0=0\}$, we have $\nu(Y_0=0)\ge\lambda(( -\infty,a_0])\ge\delta$. Since $b_0-a_0\ge L$, the inclusion $\{f\ge b_0\}\subset\{f\ge a_0+L\}\subset\{Y_0=L\}$ gives $\nu(Y_0=L)\ge\lambda([b_0,\infty))\ge\delta$. Let $\gamma:=(Y_0)_*\nu$. Then the finite measure $\gamma-\delta\delta_0-\delta\delta_{L}$ is nonnegative and has total mass $1-2\delta$. Therefore there exists a
probability measure $\eta$ on $[0,L]$ such that
\[
 \gamma
 =\delta\delta_0+\delta\delta_{L}+(1-2\delta)\eta
 =2\delta\frac{\delta_0+\delta_{L}}{2}+(1-2\delta)\eta.
\]

Let $q:=p/(p-1)$. On $Y^k$, let $Y_j(y_1,\ldots,y_k):=Y_0(y_j)$, $1\le j\le k$, and set $F_k(y_1,\ldots,y_k):=k^{-1/q}\sum_{j=1}^kY_j(y_1,\ldots,y_k)$. For $x=(x_j)$ and $y=(y_j)$, H\"older's inequality gives
\[
 \begin{aligned}
 |F_k(x)-F_k(y)|
 &\le k^{-1/q}\sum_{j=1}^k|Y_0(x_j)-Y_0(y_j)|\\
 &\le k^{-1/q}\sum_{j=1}^k d(x_j,y_j)\\
 &\le k^{-1/q}k^{1/q}
 \left(\sum_{j=1}^kd(x_j,y_j)^p\right)^{1/p}\\
 &=d_{p,k}(x,y).
 \end{aligned}
\]
Thus $F_k$ is $1$-Lipschitz.

For each coordinate, use the decomposition $\gamma=2\delta\frac{\delta_0+\delta_{L}}{2}+(1-2\delta)\eta$ as follows. Independently for each coordinate, with probability $2\delta$ choose a
value from $(\delta_0+\delta_{L})/2$, and with probability $1-2\delta$ choose a
value according to $\eta$. For each $j=1,\ldots,k$, let $I_j$ be the random
variable which equals $1$ if the first alternative is chosen in the $j$-th coordinate and equals $0$ otherwise. Thus the variables $I_1,\ldots,I_k$ are independent and $\Pr(I_j=1)=2\delta, \Pr(I_j=0)=1-2\delta$. Define $R:=\sum_{j=1}^k I_j$. Thus $R$ is the number of coordinates for which the first alternative is chosen,
and it has the binomial distribution with parameters $k$ and $2\delta$, which we
denote by $R\sim\operatorname{Bin}(k,2\delta)$. Fix the values of $I_1,\ldots,I_k$, and also fix the values chosen according
to $\eta$ in all coordinates for which $I_j=0$. Put $J:=\{j\in\{1,\ldots,k\}\mid I_j=1\}$. Then $|J|=R$. The only choices that remain are, independently for $j\in J$,
between $0$ and $L$, each with probability $1/2$. After the choices above have
been fixed, put $A_\eta:=\sum_{\{j\mid I_j=0\}}Y_j$. Thus $A_\eta$ is fixed. Therefore, for $S_k:=\sum_{j=1}^kY_j$, we have $S_k=A_\eta+L B_R$, where $B_R$ has the binomial distribution with parameters $R$ and $1/2$: $B_R\sim\operatorname{Bin}(R,1/2)$. Since the possible values of $L B_R$ are $0,L,2L,\ldots,R L$, any interval
of diameter strictly smaller than $L$ contains at most one of them. Therefore,
for every interval $I\subset\R$ with $\diam I<L$, after the choices above
have been fixed, the probability that $S_k\in I$ is at most $\max_{0\le m\le R}\Pr(B_R=m)$.

By Stirling's formula, there exists a universal constant $C_{\mathrm{bin}}>0$
such that the largest atom of the symmetric binomial law satisfies
\[
 \max_{0\le m\le R}\Pr(B_R=m)
 =2^{-R}\binom{R}{\lfloor R/2\rfloor}
 \le \frac{C_{\mathrm{bin}}}{\sqrt{R+1}}
 \qquad(R\ge0).
\]
Averaging over the preceding choices therefore gives
\[
 \sup_{\diam I<L}\Pr(S_k\in I)
 \le C_{\mathrm{bin}}\,\mathbb E(R+1)^{-1/2}.
\]
Since $R\sim\operatorname{Bin}(k,2\delta)$, the elementary identity
\[
 \mathbb E\frac1{R+1}
 =\sum_{j=0}^k\frac1{j+1}\binom{k}{j}(2\delta)^j(1-2\delta)^{k-j}
 =\frac{1-(1-2\delta)^{k+1}}{2\delta(k+1)}
 \le\frac1{2\delta(k+1)}
\]
follows by using $\binom{k}{j}/(j+1)=\binom{k+1}{j+1}/(k+1)$.
Cauchy--Schwarz now yields
\[
 \mathbb E(R+1)^{-1/2}
 \le\left(\mathbb E(R+1)^{-1}\right)^{1/2}
 \le\frac{C_\delta}{\sqrt{k+1}}.
\]
Choose $k=k(\kappa,\kappa_0)$ so large that $C_{\mathrm{bin}}C_\delta/\sqrt{k+1}<1-\kappa$. Thus,
\[
\sup_{\diam I<L}(S_k)_*\nu^{\otimes k}(I)
= \sup_{\diam I<L}\Pr(S_k\in I) < 1-\kappa.
\]

Let $B\subset\R$ be a Borel set with $(F_k)_*\nu^{\otimes k}(B)\ge1-\kappa$. If $\diam B<k^{-1/q}L$, then the set $k^{1/q}B$ has diameter smaller than $L$ and is contained in the interval $I_B:=[\inf(k^{1/q}B),\sup(k^{1/q}B)]$, which also has diameter smaller than $L$. Since $k^{1/q}F_k=S_k$, we would have
\[
 1-\kappa\le (F_k)_*\nu^{\otimes k}(B) = (S_k)_*\nu^{\otimes k}(k^{1/q}B)
\le (S_k)_*\nu^{\otimes k}(I_B) < 1-\kappa.
\]
a contradiction. Hence every such $B$ has $\diam B\ge k^{-1/q}L$. Therefore $\PartDiam((F_k)_*\nu^{\otimes k};1-\kappa) \ge k^{-1/q}L$. Since $F_k$ is $1$-Lipschitz, $\ObsDiam(Y^k,d_{p,k},\nu^{\otimes k};-\kappa) \ge k^{-1/q}L$. The integer $k=k(\kappa,\kappa_0)$ chosen above is independent of $L$.
Letting $L\uparrow D$ gives $\ObsDiam(Y^k,d_{p,k},\nu^{\otimes k};-\kappa) \ge k^{-1/q}D$. Thus the assertion holds with $c:=k^{-1/q}$.
\end{proof}

\subsection{Majority lower bound}

\begin{lem}\label{lem:majority}
Fix $1\le p<\infty$ and $0<\kappa_0<1/2$. Put
\[
 \beta_{p,*}:=\frac1{2p}-\frac1{2(\alpha_*+1)}.
\]
If $\beta_{p,*}\ge0$, then there exists $c_{p,\kappa_0}>0$ such that, for all sufficiently large $n$,
\[
 \ObsDiam(X_{n,p};-\kappa_0)\ge c_{p,\kappa_0} n^{\beta_{p,*}}.
\]
If $\beta_{p,*}<0$, then there exists $c_{\kappa_0}>0$ such that
\[
 \ObsDiam(X_{n,p};-\kappa_0)\ge c_{\kappa_0}
 \qquad(n\ge1).
\]
\end{lem}

\begin{proof}
Reflection $x\mapsto1-x$ is an isometry of $[0,1]$ and interchanges the two
sides of the interior zero. Applying it in every coordinate gives an
isomorphic product mm-space. We may therefore assume without loss of
generality that
\begin{equation}\label{eq:majority-right-max}
 \alpha_+=\alpha_*.
\end{equation}
Let $Y_1,\ldots,Y_n$ be the coordinate random variables on
$([0,1]^n,\mu^{\otimes n})$, that is,
\[
 Y_i(x_1,\ldots,x_n):=x_i
 \qquad(1\le i\le n).
\]
Let $\theta:=\mu((a,1])\in(0,1)$ and define
\[
 N_n:=\sum_{i=1}^n\mathbf 1_{\{Y_i>a\}}.
\]
Then $N_n$ has distribution $\operatorname{Bin}(n,\theta)$. Let
$m_n:=\min\{m\in\mathbb Z\mid\Pr(N_n\le m)>1/2\}$.
Then $m_n$ is a median of $N_n$ in the usual sense, and Hamza's uniform bound
\cite{Hamza}*{Theorem~2} gives
\begin{equation}\label{eq:bin-med}
 |m_n-n\theta|<\log 2.
\end{equation}
Put $\sigma_\theta:=\sqrt{\theta(1-\theta)}$ and
\[
 C_\theta
 :=C_{\mathrm{BE}}
 \frac{\theta(1-\theta)^3+(1-\theta)\theta^3}{\sigma_\theta^3}.
\]
Theorem~\ref{thm:BE}, applied to the independent random variables
$\mathbf 1_{\{Y_i>a\}}-\theta$, gives
\[
 \sup_{z\in\R}
 \left|
 \Pr\left(\frac{N_n-n\theta}{\sigma_\theta\sqrt n}\le z\right)
 -\Phi(z)
 \right|
 \le\frac{C_\theta}{\sqrt n}.
\]

Choose $c_1>0$ so small that $1-\Phi\left(\frac{2c_1}{\sigma_\theta}\right)>\kappa_0$. Decrease $c_1$ if necessary and then choose $\varepsilon>0$ such that $1-\Phi\left(\frac{2c_1}{\sigma_\theta}\right) >\kappa_0+3\varepsilon$. For each $n$, put $\ell_n:=\lceil c_1\sqrt n\rceil$. Then $c_1\sqrt n\le\ell_n<c_1\sqrt n+1$, so $\ell_n-1<c_1\sqrt n$. By \eqref{eq:bin-med}, for all sufficiently large $n$,
\[
 \frac{m_n+\ell_n-1-n\theta}{\sigma_\theta\sqrt n}
 <\frac{m_n+c_1\sqrt n-n\theta}{\sigma_\theta\sqrt n}
 \le\frac{2c_1}{\sigma_\theta}.
\]
Since $N_n$ is integer-valued, $\{N_n\ge m_n+\ell_n\} =\{N_n\le m_n+\ell_n-1\}^c$. Therefore Theorem~\ref{thm:BE} gives
\[
 \begin{aligned}
 \Pr(N_n\ge m_n+\ell_n)
 &=1-\Pr(N_n\le m_n+\ell_n-1)\\
 &\ge1-
 \Phi\left(\frac{m_n+\ell_n-1-n\theta}
 {\sigma_\theta\sqrt n}\right)-\frac{C_\theta}{\sqrt n}\\
 &\ge1-\Phi\left(\frac{2c_1}{\sigma_\theta}\right)
 -\frac{C_\theta}{\sqrt n}\\
 &>\kappa_0+2\varepsilon
 \end{aligned}
\]
for all sufficiently large $n$.

Set $H_n:=\{x\in[0,1]^n\mid N_n(x)\le m_n\}$ and $f_n(x):=d_p(x,H_n)$. $f_n$ is $1$-Lipschitz. The set $H_n$ is closed because its complement $\{N_n\ge m_n+1\}$ is open. Hence $f_n=0$ exactly on $H_n$. Also,
by the median property, $\mu^{\otimes n}(H_n)=\Pr(N_n\le m_n)\ge\frac12$.

Choose $c_2>0$, to be fixed below, and put $r_n:=c_2n^{-1/(2(\alpha_*+1))}$. For all sufficiently large $n$, $r_n<\delta$. By \eqref{eq:majority-right-max}, the right-side exponent is $\alpha_*$. Define $C_3:=\frac{C_+}{\alpha_*+1}$. By \eqref{eq:power},
\[
 \begin{aligned}
 \mu((a,a+r_n])
 &=\int_a^{a+r_n}\rho(x)\,dx\\
 &\le C_+\int_0^{r_n}s^{\alpha_*}\,ds
 =\frac{C_+}{\alpha_*+1}r_n^{\alpha_*+1}\\
 &=C_3c_2^{\alpha_*+1}n^{-1/2}.
 \end{aligned}
\]
Let $M_n(x):=\#\{i\in\{1,\ldots,n\}\mid a<x_i\le a+r_n\}$. Then $\mathbb EM_n=n\mu((a,a+r_n]) \le C_3c_2^{\alpha_*+1}\sqrt n$. Choose $c_2>0$ so small that $\frac{2C_3c_2^{\alpha_*+1}}{c_1}<\varepsilon$. Markov's inequality then gives
\[
 \Pr\left(M_n>\frac{c_1}{2}\sqrt n\right)
 \le\frac{2\mathbb EM_n}{c_1\sqrt n}
 <\varepsilon.
\]
Consequently the event
\[
 E_n:=
 \left\{N_n\ge m_n+\ell_n\right\}
 \cap
 \left\{M_n\le\frac{c_1}{2}\sqrt n\right\}
\]
satisfies
\[
 \mu^{\otimes n}(E_n)
 \ge \Pr(N_n\ge m_n+\ell_n)
 -\Pr\left(M_n>\frac{c_1}{2}\sqrt n\right)
 >\kappa_0+\varepsilon>\kappa_0.
\]

Fix $x\in E_n$ and $y\in H_n$. Let $A_x:=\{i\mid x_i>a\}$, $A_y:=\{i\mid y_i>a\}$. Then $|A_x|-|A_y|=N_n(x)-N_n(y)\ge\ell_n\ge c_1\sqrt n$. Since
\[
 |A_x|-|A_y|
 =|A_x\setminus A_y|-|A_y\setminus A_x|
 \le|A_x\setminus A_y|,
\]
we have $|A_x\setminus A_y|\ge c_1\sqrt n$. For every $i\in A_x\setminus A_y$, one has $x_i>a$ and $y_i\le a$.
Among these indices, at most $M_n(x)\le(c_1/2)\sqrt n$ satisfy $a<x_i\le a+r_n$. Therefore at least $(c_1/2)\sqrt n$ indices satisfy
simultaneously $x_i>a+r_n$, $y_i\le a$. For each such index, $|x_i-y_i|\ge r_n$. It follows that $d_p(x,y)^p =\sum_{i=1}^n|x_i-y_i|^p \ge\frac{c_1}{2}\sqrt n\,r_n^p$. Define $c_4:=\left(\frac{c_1}{2}\right)^{1/p}c_2$. Taking the infimum over $y\in H_n$ yields
\[
 f_n(x)
 \ge\left(\frac{c_1}{2}\right)^{1/p}n^{1/(2p)}r_n
 =c_4n^{\frac1{2p}-\frac1{2(\alpha_*+1)}}
 =c_4n^{\beta_{p,*}}.
\]
Thus
\[
 \mu^{\otimes n}\bigl(\{f_n\ge c_4n^{\beta_{p,*}}\}\bigr)
 \ge\mu^{\otimes n}(E_n)>\kappa_0.
\]
On the other hand, $\mu^{\otimes n}(\{f_n=0\}) =\mu^{\otimes n}(H_n)\ge\frac12>\kappa_0$. Let $B\subset\R$ be Borel and suppose $(f_n)_*\mu^{\otimes n}(B)\ge1-\kappa_0$. If $0\notin B$, then $\{f_n=0\}$ is contained in the complement of $f_n^{-1}(B)$, whose measure would therefore be at least $1/2>\kappa_0$,
a contradiction. Hence $0\in B$. If $B\cap[c_4n^{\beta_{p,*}},\infty)=\varnothing$, then $\{f_n\ge c_4n^{\beta_{p,*}}\}$ is contained in the complement of $f_n^{-1}(B)$, whose measure would be greater than $\kappa_0$, again a
contradiction. Hence $B$ meets $[c_4n^{\beta_{p,*}},\infty)$.
Therefore $\diam B\ge c_4n^{\beta_{p,*}}$. Taking the infimum over all such $B$ gives
\[
 \PartDiam((f_n)_*\mu^{\otimes n};1-\kappa_0)
 \ge c_4n^{\beta_{p,*}}.
\]
Since $f_n$ is $1$-Lipschitz, $\ObsDiam(X_{n,p};-\kappa_0) \ge c_4n^{\beta_{p,*}}$.

Suppose finally that $\beta_{p,*}<0$. The coordinate map
$P_1(x_1,\ldots,x_n):=x_1$ is $1$-Lipschitz because
$|x_1-y_1|\le d_p(x,y)$. By Lemma~\ref{lem:positive-partdiam},
$\PartDiam(\mu;1-\kappa_0)>0$, and the coordinate projection gives the
asserted dimension-independent lower bound.
\end{proof}

\subsection{Critical lower bound}

Assume now $1<p\le2$, $q=p/(p-1)$, and
$\alpha_*=q-1=1/(p-1)$. By reflection if necessary, we may assume
\begin{equation}\label{eq:critical-right}
 \alpha_+=\alpha_*=q-1.
\end{equation}
For $0<r<\delta$, put $L_r:=\log(\delta/r)$ and define
$g_r:[0,1]\to[-1,1]$ by
\[
 g_r(x):=
 \begin{cases}
 -1,&0\le x\le a+r,\\[1mm]
 -1+\dfrac{2\log((x-a)/r)}{L_r},&a+r<x<a+\delta,\\[3mm]
 1,&a+\delta\le x\le1.
 \end{cases}
\]
Set $h_r:=g_r-\int_0^1g_r\,d\mu$.

\begin{lem}\label{lem:var-ramp}
There exist $c>0$ and $r_0>0$ such that $\Var_\mu(h_r)\ge c$ for
$0<r<r_0$.
\end{lem}

\begin{proof}
Since $h_r$ differs from $g_r$ by a constant,
$\Var_\mu(h_r)=\Var_\mu(g_r)$. We prove that as $r \downarrow 0$, $g_r$ converges pointwise
away from $a$ to
\[
 g_0(x):=
 \begin{cases}
 -1,&0\le x<a,\\
 1,&a<x\le1.
 \end{cases}
\]
For $x<a$, one has $g_r(x)=-1$ for every $r>0$. If
$a<x<a+\delta$, then for all sufficiently small $r<x-a$,
\[
 g_r(x)
 =-1+2\frac{\log((x-a)/r)}{\log(\delta/r)}
 =-1+2\frac{\log(1/r)+\log(x-a)}{\log(1/r)+\log\delta}
 \longrightarrow1.
\]
If $a+\delta\le x\le1$, then $g_r(x)=1$. Thus $g_r(x)\to g_0(x)$ for
every $x\ne a$. Because $|g_r|\le1$ and $\mu(\{a\})=0$, dominated
convergence gives convergence of the first and second moments. Put
$\lambda_-:=\mu([0,a))$ and $\lambda_+:=\mu((a,1])$. Both are positive,
$\lambda_-+\lambda_+=1$, and
$\int_{[0,1]} g_0\,d\mu=-\lambda_-+\lambda_+$. Hence
\[
 \Var_\mu(g_0)=1-(\lambda_+-\lambda_-)^2=4\lambda_-\lambda_+>0.
\]
Therefore $\Var_\mu(h_r)=\Var_\mu(g_r)\to4\lambda_-\lambda_+$, which
proves the assertion.
\end{proof}

For $\tau>0$ and bounded $h:[0,1]\to\R$, define
\[
 Q_\tau h(x):=\inf_{y\in[0,1]}
 \left\{h(y)+\frac{|x-y|^p}{\tau^{p-1}}\right\}.
\]

\begin{lem}\label{lem:hopf}
There exist $c,C,r_0>0$ such that, if $0<r<r_0$ and
$0<\tau\le cr^qL_r^{q-1}$, then
\[
 0\le\int_0^1(h_r-Q_\tau h_r)\,d\mu\le C\tau L_r^{1-q}
\]
and
\[
 0\le h_r(x)-Q_\tau h_r(x)\le C\frac{\tau}{r^qL_r^q}
 \qquad(x\in[0,1]).
\]
\end{lem}

\begin{proof}
Since $h_r-g_r$ is constant and the Hopf--Lax operator commutes with
addition of constants, $h_r-Q_\tau h_r=g_r-Q_\tau g_r$. We first record
two estimates for an arbitrary $K$-Lipschitz function $h$. For
$s:=|x-y|$, the inequality $h(x)-h(y)\le Ks$ gives
\[
 h(x)-Q_\tau h(x)
 \le\sup_{s\ge0}\left(Ks-\frac{s^p}{\tau^{p-1}}\right)
 =(p-1)p^{-q}\tau K^q.
\]
Thus
\begin{equation}\label{eq:HL-global}
 0\le h(x)-Q_\tau h(x)\le(p-1)p^{-q}\tau K^q.
\end{equation}
Moreover, the minimization defining $Q_\tau h(x)$ is local. Let
$s:=|x-y|$. If $s\ge\tau K^{q-1}$, then
$s^p/\tau^{p-1} \ge Ks$.
Since $h(y)\ge h(x)-Ks$, it follows that
\[
 h(y)+\frac{|x-y|^p}{\tau^{p-1}}
 \ge h(x)-Ks+\frac{s^p}{\tau^{p-1}}
 \ge h(x).
\]
Since the choice $y=x$ yields the value $h(x)$ in the infimum defining
$Q_\tau h(x)$, such points $y$ may be omitted. Hence
\begin{equation}\label{eq:HL-local}
 Q_\tau h(x)
 =
 \inf_{\substack{y\in[0,1]\\ |x-y|\le\tau K^{q-1}}}
 \left\{h(y)+\frac{|x-y|^p}{\tau^{p-1}}\right\}.
\end{equation}

For $g_r$, $K_r:=\Lip(g_r)=2/(rL_r)$. The localization radius is
$R_r:=\tau K_r^{q-1}$. If $\tau\le cr^qL_r^{q-1}$ and $c$ is sufficiently
small, then $R_r\le r/4$. Equation~\eqref{eq:HL-global} gives the global
pointwise bound
\[
 0\le g_r(x)-Q_\tau g_r(x)\le C_0\frac{\tau}{r^qL_r^q},
 \qquad C_0:=(p-1)p^{-q}2^q.
\]

It remains to estimate the integral. If $x\le a+r$, then
$g_r(x)=-1=\min g_r$. Hence, for every $y\in[0,1]$,
\[
 g_r(y)+\frac{|x-y|^p}{\tau^{p-1}}\ge-1,
\]
whereas the choice $y=x$ gives the value $g_r(x)=-1$. Thus
$Q_\tau g_r(x)=g_r(x)$ and
\[
 \int_{[0,a+r]}(g_r-Q_\tau g_r)\,d\mu=0.
\]
On $I_r:=(a+r,a+2r]$, the critical right-side bound in \eqref{eq:power}
and \eqref{eq:critical-right} gives a constant $C_1>0$, independent of
$r$, such that $\mu(I_r)\le C_1r^q$. Therefore the global pointwise
estimate above yields
\[
 \int_{I_r}(g_r-Q_\tau g_r)\,d\mu
 \le \mu(I_r)C_0\frac{\tau}{r^qL_r^q}
 \le C_0C_1\tau L_r^{-q}.
\]

Now let $a+2r<x<a+2\delta$ and put $u:=x-a$. By \eqref{eq:HL-local},
we may restrict the infimum to
$|x-y|\le R_r\le r/4<u/8$. For such a $y$, every point $z$ on the
segment joining $x$ to $y$ satisfies
\[
 \frac78u<u-\frac r4\le z-a\le u+\frac r4<\frac98u.
\]
On the logarithmic part of $g_r$ we therefore have
$|g_r'(z)|=2/((z-a)L_r)\le C/(uL_r)$, while the derivative is zero on
the constant parts. Thus $|g_r'|\le K_{x,r}:=C/(uL_r)$ for
$\cL^1$-a.e.\ $z$ on the segment. Since $g_r$ is absolutely continuous,
for $s:=|x-y|$ we obtain
\[
 |g_r(x)-g_r(y)|\le K_{x,r}s,
 \qquad\text{and hence}\qquad
 g_r(x)-g_r(y)\le K_{x,r}s.
\]
Consequently the same one-variable maximization used in
\eqref{eq:HL-global}, with $y$ restricted as in \eqref{eq:HL-local}, gives
\[
 0\le g_r(x)-Q_\tau g_r(x)
 \le \sup_{s\ge0}
 \left(K_{x,r}s-\frac{s^p}{\tau^{p-1}}\right)
 =(p-1)p^{-q}\tau K_{x,r}^q
 \le C\frac{\tau}{u^qL_r^q}.
\]
Writing $x=a+u$, we have $d\mu(x)=\rho(a+u)\,du$. Since
$\rho(a+u)\le C_+u^{q-1}$ for $\cL^1$-a.e.\ $0<u<2\delta$, the preceding
pointwise bound gives
\[
 \int_{a+2r}^{a+2\delta}(g_r-Q_\tau g_r)\,d\mu
 \le\frac{C\tau}{L_r^q}\int_{2r}^{2\delta}
 \frac{\rho(a+u)}{u^q}\,du
 \le\frac{C\tau}{L_r^q}\int_{2r}^{2\delta}\frac{du}{u}
 \le C\tau L_r^{1-q}.
\]
Finally, if $x\ge a+2\delta$, every relevant $y$ lies to the right of
$a+\delta$ for all sufficiently small $r$: indeed,
$y\ge x-R_r\ge a+2\delta-r/4\ge a+\delta$. Hence
$g_r(y)=g_r(x)=1$ for every relevant $y$. The restricted infimum is
therefore $1$, and thus
\[
 \int_{[a+2\delta,1]}(g_r-Q_\tau g_r)\,d\mu=0.
\]
Combining the three regions gives
\[
 \int_0^1(h_r-Q_\tau h_r)\,d\mu
 =\int_0^1(g_r-Q_\tau g_r)\,d\mu
 \le C_0C_1\tau L_r^{-q}+C\tau L_r^{1-q}
 \le C\tau L_r^{1-q},
\]
where the last inequality uses $L_r^{-q}\le L_r^{1-q}$ for $L_r\ge1$.
This proves the integral estimate.
\end{proof}

The following is the converse part of \cite{GRS}*{Lemma~1.7}.
\begin{prop}[Observable diameter and enlargement]\label{prop:OD-enlarge}
Let $(X,d,\nu)$ be an mm-space, $0<\kappa<1/2$, and
\[
 R>\ObsDiam(X,d,\nu;-\kappa).
\]
Then every Borel set $A\subset X$ with $\nu(A)\ge1/2$ satisfies
\[
 \nu(A_R)\ge1-\kappa,
 \qquad
 A_R:=\{x\in X\mid d(x,A)<R\}.
\]
\end{prop}

\begin{prop}\label{prop:critical-lower}
Assume $\alpha_*=1/(p-1)$.
For every fixed $0<\kappa<1$, there exists
$c_\kappa>0$ such that
\[
 \ObsDiam(X_{n,p};-\kappa)
 \ge c_\kappa n^{1/p-1/2}(\log n)^{1/p}
\]
for every $n$.
\end{prop}

\begin{proof}
We first prove the estimate at one fixed level $0<\kappa_0<1/2$. For all sufficiently large $n$, put $r_n:=n^{-1/(2q)}$; then $r_n<\delta$. Define $L_n:=\log(\delta/r_n)$, $t_n:=n^{-1/(2(p-1))}$, and $s_n:=\varepsilon^{p-1}n^{(2-p)/2}L_n$, where $\varepsilon>0$ will be chosen later. Put $h_n:=h_{r_n}$ and $\tau_n:=t_ns_n^{1/(p-1)}$. Since $1/(p-1)=q-1$,
\[
 \begin{aligned}
 \tau_n
 &=n^{-1/(2(p-1))}
 \left(\varepsilon^{p-1}n^{(2-p)/2}L_n\right)^{1/(p-1)}\\
 &=\varepsilon
 n^{-1/(2(p-1))+(2-p)/(2(p-1))}
 L_n^{1/(p-1)}\\
 &=\varepsilon n^{-1/2}L_n^{q-1}.
 \end{aligned}
\]
Also $r_n^q=n^{-1/2}$, so $\tau_n=\varepsilon r_n^qL_n^{q-1}$. Assume for now that $0<\varepsilon<c$, where $c$ is the constant in
Lemma~\ref{lem:hopf}. Then Lemma~\ref{lem:hopf} applies for all
sufficiently large $n$.

Define $q_n:=Q_{\tau_n}h_n$ and $\delta_n:=\int_0^1(h_n-q_n)\,d\mu$. The integral estimate in Lemma~\ref{lem:hopf} gives $\delta_n \le C\tau_nL_n^{1-q} =C\varepsilon n^{-1/2}$, so
\begin{equation}\label{eq:critical-bias}
 \sqrt n\,\delta_n\le C\varepsilon.
\end{equation}
The pointwise conclusion of Lemma~\ref{lem:hopf}, applied with
$r=r_n$ and $\tau=\tau_n$, gives
\[
 \|h_n-q_n\|_\infty
 \le C\frac{\tau_n}{r_n^qL_n^q}
 =C\frac{\varepsilon}{L_n}
 \longrightarrow0.
\]
By Lemma~\ref{lem:var-ramp}, there is $v_0>0$ such that $\Var_\mu(h_n)\ge v_0$ for all sufficiently large $n$. Let $Y$ be a random variable with distribution $\mu$. We now show that the variances of $q_n$ are also uniformly positive. For square-integrable random variables $U,V$,
\[
 \left|\sqrt{\Var(U)}-\sqrt{\Var(V)}\right|
 \le\|(U-\mathbb EU)-(V-\mathbb EV)\|_2.
\]
Moreover
\[
 \|(U-\mathbb EU)-(V-\mathbb EV)\|_2
 \le\|U-V\|_2+|\mathbb E(U-V)|
 \le2\|U-V\|_2.
\]
Applying this with $U=q_n(Y)$ and $V=h_n(Y)$ gives
\[
 \left|\sqrt{\Var(q_n)}-\sqrt{\Var(h_n)}\right|
 \le2\|q_n-h_n\|_{L^2(\mu)}
 \le2\|q_n-h_n\|_\infty\longrightarrow0.
\]
Consequently, for all sufficiently large $n$,
\[
 \sqrt{\Var_\mu(q_n)}
 \ge \sqrt{v_0}-2\|q_n-h_n\|_\infty
 \ge \frac{\sqrt{v_0}}2.
\]
Thus, with $\sigma_*:=\sqrt{v_0}/2$, which is independent of
$\varepsilon$, we have
\begin{equation}\label{eq:critical-var}
 \Var_\mu(q_n)\ge\sigma_*^2
\end{equation}
for all sufficiently large $n$.

Let $Y_1,\ldots,Y_n$ be independent copies of $Y$, and put $S_n:=\sum_{i=1}^nh_n(Y_i)$ and $\sigma_{h,n}^2:=\Var_\mu(h_n)$. For every $i=1,\ldots,n$,
\[
 \mathbb E h_n(Y_i)=0,
 \qquad |h_n(Y_i)|\le 2,
 \qquad v_0\le \Var\bigl(h_n(Y_i)\bigr)=\sigma_{h,n}^2\le 4.
\]
Hence Theorem~\ref{thm:BE} gives
\[
 \sup_{z\in\R}\left|\Pr\left(\frac{S_n}{\sqrt n\,\sigma_{h,n}}\le z\right)-\Phi(z)\right|\le\frac{C}{\sqrt n}.
\]
Let $M_n$ be the lower median of $S_n$ and put $z_n:=M_n/(\sqrt n\,\sigma_{h,n})$. From $\Pr(S_n\le M_n)\ge1/2$ and $\Pr(S_n<s)<1/2$ for every $s<M_n$, letting $s\uparrow M_n$ and using the continuity of $\Phi$ yields
\[
 \frac12-\frac{C}{\sqrt n}\le\Phi(z_n)\le\frac12+\frac{C}{\sqrt n}.
\]
Since $\Phi$ is continuous and strictly increasing with $\Phi(0)=1/2$, we conclude that
\begin{equation}\label{eq:critical-median}
 z_n=\frac{M_n}{\sqrt n\,\sigma_{h,n}}\longrightarrow0.
\end{equation}

Define $F_n(x_1,\ldots,x_n):=s_n\sum_{i=1}^nh_n(x_i)$ and put $R_n:=\ObsDiam(X_{n,p};-\kappa_0)$. Let $A_n:=\left\{x\in[0,1]^n\mid\sum_{i=1}^nh_n(x_i)\le M_n\right\}$. By the definition of $M_n$, $\mu^{\otimes n}(A_n)\ge\frac12$. For $\zeta>0$, $R_n+\zeta>R_n$, so Proposition~\ref{prop:OD-enlarge}
gives
\begin{equation}\label{eq:critical-enlarge}
 \mu^{\otimes n}((A_n)_{R_n+\zeta})\ge1-\kappa_0.
\end{equation}

Define the \emph{product inf-convolution}
\[
 Q_{t_n}^{(n)}F_n(x)
 :=\inf_{y\in[0,1]^n}
 \left\{F_n(y)+\frac{d_p(x,y)^p}{t_n^{p-1}}\right\}.
\]
If $x\in(A_n)_{R_n+\zeta}$, there exists $y\in A_n$ satisfying $d_p(x,y)<R_n+\zeta$. Since $y\in A_n$, $F_n(y)=s_n\sum_{i=1}^nh_n(y_i)\le s_nM_n$. Substituting this $y$ into the definition of $Q_{t_n}^{(n)}F_n(x)$ gives $Q_{t_n}^{(n)}F_n(x) \le s_nM_n+\frac{(R_n+\zeta)^p}{t_n^{p-1}}$. Together with \eqref{eq:critical-enlarge},
\begin{equation}\label{eq:critical-prob-upper}
 \Pr\left(
 Q_{t_n}^{(n)}F_n
 >s_nM_n+\frac{(R_n+\zeta)^p}{t_n^{p-1}}
 \right)
 \le\kappa_0.
\end{equation}

We now compute the product inf-convolution exactly. Since $d_p(x,y)^p=\sum_{i=1}^n|x_i-y_i|^p$, we have
\begin{align}\label{eq:critical-separation}
 Q_{t_n}^{(n)}F_n(x)
 &=\inf_{y_1,\ldots,y_n\in[0,1]}
 \sum_{i=1}^n
 \left(s_nh_n(y_i)+\frac{|x_i-y_i|^p}{t_n^{p-1}}\right)\\
 \notag &=\sum_{i=1}^n
 \inf_{y_i\in[0,1]}
 \left(s_nh_n(y_i)+\frac{|x_i-y_i|^p}{t_n^{p-1}}\right)\\
 \notag &=s_n \sum_{i=1}^n \inf_{y_i\in[0,1]}
 \left(h_n(y_i)+\frac{|x_i-y_i|^p}{\tau_n^{p-1}}\right)\\
 \notag &=s_n\sum_{i=1}^nq_n(x_i).
\end{align}

Since $\int_0^1 h_n\,d\mu=0$ and $\delta_n=\int_0^1(h_n-q_n)\,d\mu$, we have $\mathbb E q_n(Y_i)=-\delta_n$. Put $\sigma_n^2:=\Var_\mu(q_n)$ and
\[
 Z_n:=\frac{\sum_{i=1}^n(q_n(Y_i)-\mathbb Eq_n(Y_i))}
 {\sqrt n\,\sigma_n}.
\]
Combining \eqref{eq:critical-prob-upper} and
\eqref{eq:critical-separation}, the event in
\eqref{eq:critical-prob-upper} is
\[
 \left\{
 s_n\sum_{i=1}^nq_n(Y_i)
 >s_nM_n+\frac{(R_n+\zeta)^p}{t_n^{p-1}}
 \right\}.
\]
Since $\sum_{i=1}^nq_n(Y_i) =\sqrt n\,\sigma_n Z_n-n\delta_n$, dividing by $s_n\sqrt n\,\sigma_n$ shows that $\Pr(Z_n>B_n)\le\kappa_0$, where
\[
 B_n:=
 \frac{M_n}{\sqrt n\,\sigma_n}
 +\frac{(R_n+\zeta)^p}
 {s_nt_n^{p-1}\sqrt n\,\sigma_n}
 +\frac{\sqrt n\,\delta_n}{\sigma_n}.
\]

The variables $q_n(Y_i)$ are uniformly bounded because $\|q_n-h_n\|_\infty\to0$ and $\|h_n\|_\infty\le2$. By
\eqref{eq:critical-var}, their variances are uniformly bounded below. Thus
Theorem~\ref{thm:BE} again gives
\[
\sup_{z\in\R}|\Pr(Z_n\le z)-\Phi(z)|\le\frac{C}{\sqrt n}.
\]
Choose $z_0>0$ such that $1-\Phi(z_0)>\kappa_0$. Then, for all sufficiently large $n$, $\Pr(Z_n>z_0)>\kappa_0$. If $B_n<z_0$, then $\Pr(Z_n>B_n)\ge\Pr(Z_n>z_0)>\kappa_0$, contradicting $\Pr(Z_n>B_n)\le\kappa_0$. Hence $B_n\ge z_0$. Therefore
\[
 \frac{(R_n+\zeta)^p}
 {s_nt_n^{p-1}\sqrt n\,\sigma_n}
 \ge
 z_0-\frac{M_n}{\sqrt n\,\sigma_n}
 -\frac{\sqrt n\,\delta_n}{\sigma_n}.
\]
By \eqref{eq:critical-median}, $\sigma_{h,n}\le2$, and \eqref{eq:critical-var}, the middle term satisfies
\[
 \left|\frac{M_n}{\sqrt n\,\sigma_n}\right|=\frac{\sigma_{h,n}}{\sigma_n}|z_n|\le\frac{2}{\sigma_*}|z_n|\longrightarrow0.
\]
By \eqref{eq:critical-bias}, the last term is bounded by $C\varepsilon/\sigma_*$. Choose $\varepsilon>0$ sufficiently small that $\varepsilon<c$ and $\frac{C\varepsilon}{\sigma_*}\le\frac{z_0}{4}$, and then choose $n$ sufficiently large that $\left|\frac{M_n}{\sqrt n\,\sigma_n}\right|\le\frac{z_0}{4}$. Then $\frac{(R_n+\zeta)^p} {s_nt_n^{p-1}\sqrt n\,\sigma_n} \ge\frac{z_0}{2}$. Since $\sigma_n\ge\sigma_*$, $(R_n+\zeta)^p \ge\frac{z_0\sigma_*}{2}s_nt_n^{p-1}\sqrt n$. Now
\[
 t_n^{p-1}\sqrt n
 =\left(n^{-1/(2(p-1))}\right)^{p-1}\sqrt n
 =n^{-1/2}\sqrt n=1.
\]
Thus $(R_n+\zeta)^p\ge cs_n$. Letting $\zeta\downarrow0$ gives $R_n^p\ge cs_n =c\varepsilon^{p-1}n^{(2-p)/2}L_n$. Since $L_n=\log\delta+\frac1{2q}\log n \asymp\log n$, we obtain $R_n\ge c n^{(2-p)/(2p)}(\log n)^{1/p}$. Finally $\frac{2-p}{2p}=\frac1p-\frac12$, so $R_n\ge c n^{1/p-1/2}(\log n)^{1/p}$. This proves the estimate at level $\kappa_0$.

Now fix an arbitrary $0<\kappa<1$. Let $k=k(\kappa,\kappa_0)$ be given by Lemma~\ref{lem:block}, and for each $n$ put $m:=\left\lfloor\frac nk\right\rfloor$. For all sufficiently large $n$, the fixed-level conclusion established in the first part of this proof gives $\ObsDiam(X_{m,p};-\kappa_0) \ge c m^{1/p-1/2}(\log m)^{1/p}$. Applying Lemma~\ref{lem:block} directly to $X_{m,p}$ gives a constant $c'>0$,
independent of $m$, such that $\ObsDiam(X_{km,p};-\kappa) \ge c' m^{1/p-1/2}(\log m)^{1/p}$. Let $\pi_{km}:X_{n,p}\to X_{km,p}$ be the projection onto the first $km$ coordinates. It is $1$-Lipschitz and $(\pi_{km})_*\mu^{\otimes n}=\mu^{\otimes km}$. Hence every $1$-Lipschitz
function $g:X_{km,p}\to\R$ gives the $1$-Lipschitz function $g\circ\pi_{km}:X_{n,p}\to\R$ with the same push-forward measure. Taking
the supremum over such $g$ yields $\ObsDiam(X_{n,p};-\kappa) \ge \ObsDiam(X_{km,p};-\kappa)$. Since $k$ is fixed, $m\asymp n$, $\log m\asymp\log n$. Hence, for all sufficiently large $n$, $\ObsDiam(X_{n,p};-\kappa) \ge c n^{1/p-1/2}(\log n)^{1/p}$. It remains only to cover the finitely many smaller $n$. For $n=1$ the
right-hand side is zero. For each of the remaining finitely many integers $n\ge2$, the coordinate projection $P_1:X_{n,p}\to[0,1]$ is $1$-Lipschitz,
and therefore, by Lemma~\ref{lem:positive-partdiam},
\[
 \ObsDiam(X_{n,p};-\kappa)
 \ge\PartDiam(\mu;1-\kappa)>0.
\]
The quantity $n^{1/p-1/2}(\log n)^{1/p}$ is bounded on this finite set.
After decreasing $c$ once more, the asserted estimate holds for every $n\ge1$.
\end{proof}

\begin{prop}\label{prop:lower-le2}
For $1<p\le2$ and every fixed $0<\kappa<1$, as $n\to\infty$,
\[
 \ObsDiam(X_{n,p};-\kappa)
 \gtrsim
 \begin{cases}
 n^{1/p-1/2},&\alpha_*<1/(p-1),\\
 n^{1/p-1/2}(\log n)^{1/p},&\alpha_*=1/(p-1),\\
 n^{\frac1{2p}-\frac1{2(\alpha_*+1)}},&\alpha_*>1/(p-1).
 \end{cases}
\]
\end{prop}

\begin{proof}
Fix $0<\kappa<1$.

If $\alpha_*< 1/(p-1)$, Lemma~\ref{lem:baseline} gives the required subcritical lower bound.

If $\alpha_*= 1/(p-1)$, Proposition~\ref{prop:critical-lower} gives the required critical lower bound.

Suppose now that $\alpha_*> 1/(p-1)$. Fix $0<\kappa_0<1/2$. Since $1<p\le2$, the inequality $\alpha_*>1/(p-1) \ge 1$ implies $\alpha_*+1>p$, and therefore $\beta_{p,*}:=\frac1{2p}-\frac1{2(\alpha_*+1)}>0$. Lemma~\ref{lem:majority} gives, for all sufficiently large $m$, $\ObsDiam(X_{m,p};-\kappa_0) \ge c m^{\beta_{p,*}}$. Applying Lemma~\ref{lem:block} to $Y=X_{m,p}$, we obtain a fixed integer $k=k(\kappa,\kappa_0)$ and a constant $c'>0$, both independent of $m$,
such that $\ObsDiam(X_{km,p};-\kappa) \ge c'm^{\beta_{p,*}}$. For a general sufficiently large $n$, set $m:=\left\lfloor\frac nk\right\rfloor$. Then $km\le n$. Let $\pi_{km}:X_{n,p}\to X_{km,p}$ be the projection onto
the first $km$ coordinates. The map $\pi_{km}$ is $1$-Lipschitz and satisfies $(\pi_{km})_*\mu^{\otimes n}=\mu^{\otimes km}$. Therefore every $1$-Lipschitz function $g:X_{km,p}\to\R$ gives the $1$-Lipschitz function $g\circ\pi_{km}:X_{n,p}\to\R$ with the same push-forward measure. Taking
the supremum over all such $g$ gives
\[
 \ObsDiam(X_{n,p};-\kappa)
 \ge \ObsDiam(X_{km,p};-\kappa)
 \ge c'm^{\beta_{p,*}}.
\]
Since $k$ is fixed, $m\asymp n$, and hence $m^{\beta_{p,*}}\asymp n^{\beta_{p,*}}$. Thus $\ObsDiam(X_{n,p};-\kappa) \ge c n^{\beta_{p,*}}$. This is the required supercritical lower bound.
\end{proof}

\section{The range \texorpdfstring{$p>2$}{p>2}}
Throughout this section, fix $p>2$.

\subsection{Transport-entropy estimates}

The following is the case $\alpha(t)=t/C$ of
\cite{GozlanLeonard}*{Proposition~4.3}.

\begin{prop}[Tensorization]\label{prop:tensor-Tp}
Let $(X,d,\sigma)$ be an mm-space satisfying $T_p(C)$ for a constant $C>0$.
Then, for every $n\ge1$, the product measure $\sigma^{\otimes n}$ satisfies $T_p(C)$ on $X^n$ equipped with the $\ell^p$ product metric
\[
 d_p(x,y)^p:=\sum_{i=1}^n d(x_i,y_i)^p.
\]
\end{prop}

The following is \cite{GozlanLeonard}*{Theorem~1.7}.

\begin{thm}[Marton's concentration theorem]\label{thm:Marton}
Let $(X,d)$ be a metric space, let $\sigma$ be a Borel probability measure on
$X$, and let $\alpha:[0,\infty)\to[0,\infty)$ be a bijection. Suppose that
\[
 \alpha\bigl(W_1(\nu,\sigma)\bigr)\le H(\nu\mid\sigma)
\]
for every Borel probability measure $\nu$ on $X$.
Put $r_0:=\alpha^{-1}(\log2)$.
Then every Borel set $A\subset X$ with $\sigma(A)\ge1/2$ satisfies
\[
 \sigma\bigl(A^{[r]}\bigr)\ge1-\exp\bigl(-\alpha(r-r_0)\bigr)
 \qquad(r\ge r_0),
 \qquad
 A^{[r]}:=\{x\in X\mid d(x,A)\le r\}.
\]
Equivalently, if $f:X\to\R$ is $1$-Lipschitz and $m_f$ is a median of
$f_*\sigma$, then
\[
 \sigma\bigl(f>m_f+r+r_0\bigr)\le\exp\bigl(-\alpha(r)\bigr)
 \qquad(r\ge0).
\]
\end{thm}

We prove the following.

\begin{lem}\label{lem:Tp-obs}
If $(X,d,\sigma)$ satisfies $T_p(C)$ for a constant $C>0$, then for every $0<\kappa<1$,
\[
 \ObsDiam(X,d,\sigma;-\kappa)
 \le 2C^{1/p}\left((\log2)^{1/p}+\left(\log\frac2\kappa\right)^{1/p}\right).
\]
The same bound holds for all product powers with the $\ell^p$ product metric.
\end{lem}

\begin{proof}
Let $\alpha(t):=t^p/C$.
Since $W_1\le W_p$, the inequality $T_p(C)$ implies
$\alpha(W_1(\nu,\sigma))\le H(\nu\mid\sigma)$
for every probability measure $\nu$ on $X$.
We have $\alpha^{-1}(u)=C^{1/p}u^{1/p}$.
For every Borel set $A\subset X$ with $\sigma(A)\ge1/2$,
Theorem~\ref{thm:Marton} and the fact that the right-hand side below is zero
for $r<C^{1/p}(\log2)^{1/p}$ give
\begin{equation}\label{eq:Tp-concentration}
 \sigma\bigl(A^{[r]}\bigr)\ge
 1-\exp\left[-\frac{(r-C^{1/p}(\log2)^{1/p})_+^p}{C}\right].
\end{equation}
Let $f:X\to\R$ be $1$-Lipschitz and let $m$ be a median of
$f_*\sigma$. Then both $\{f\le m\}$ and $\{f\ge m\}$ have
$\sigma$-measure at least $1/2$. Moreover,
\[
 \{f>m+r\}\subset X\setminus\{f\le m\}^{[r]},
 \qquad
 \{f<m-r\}\subset X\setminus\{f\ge m\}^{[r]}.
\]
Applying \eqref{eq:Tp-concentration} with $A=\{f\le m\}$ and
$A=\{f\ge m\}$ gives
\[
 \sigma\bigl(f>m+r\bigr),\quad \sigma\bigl(f<m-r\bigr)
 \le
 \exp\left[-\left(\frac r{C^{1/p}}-(\log2)^{1/p}\right)_+^p\right].
\]
Taking
\[
 r=C^{1/p}\left((\log2)^{1/p}+\left(\log\frac2\kappa\right)^{1/p}\right)
\]
makes each tail at most $\kappa/2$. Hence
$\sigma(m-r\le f\le m+r)\ge1-\kappa$, and therefore
$\PartDiam(f_*\sigma;1-\kappa)\le2r$. Taking the supremum over all
$1$-Lipschitz $f$ proves the required bound of the observable diameter.

By Proposition~\ref{prop:tensor-Tp}, every product power satisfies the same
$T_p(C)$ inequality for the $\ell^p$ product metric. Hence the
observable-diameter bound holds for all product powers.
\end{proof}

\subsection{Mass-shift estimate} \label{ssec:Mass-shift}
Recall that $d\mu_{\mathbf r}:=\rho_{\mathbf r}\,d\cL^1$, where $\rho_{\mathbf r}$ is the two-sided regularization in Definition~\ref{defn:two-reg}.
Let
\[
 \lambda:=\mu_{\mathbf r}([0,a])=\mu([0,a])\in(0,1),
\]
and put
\[
 \mu_{\mathbf r,-}:=\lambda^{-1}\mu_{\mathbf r}|_{[0,a]},
 \qquad
 \mu_{\mathbf r,+}:=(1-\lambda)^{-1}\mu_{\mathbf r}|_{[a,1]}.
\]
For $\omega\in\{-,+\}$, set
\[
 m_\omega:=r_\omega^{\alpha_\omega+1},
 \qquad
 A_\omega:=\max\{1,r_\omega^{p-1-\alpha_\omega}\},
\]
and define
\[
 A_{\mathbf r}:=\max\{A_-,A_+\}
 =\max\{1,r_-^{p-1-\alpha_-},r_+^{p-1-\alpha_+}\}.
\]

For a Borel probability measure $\eta$ on $\R$, write
\[
 F_\eta(x):=\eta((-\infty,x]),
 \qquad
 F_\eta^{-1}(t):=\inf\{x\in\R\mid F_\eta(x)\ge t\}
 \quad(0<t<1).
\]

The following is proved in \cite{GozlanLine}*{Theorem~1.2}.

\begin{thm}[Monotone optimal transport]\label{thm:monotone}
Let $\sigma,\nu$ be Borel probability measures on $\R$, assume that $\sigma$ has no atoms, and let $\beta:[0,\infty)\to[0,\infty)$ be convex with $\beta(0)=0$. Assume that
\[
 \int_{\R^2}\beta(|x-y|)\,d\sigma(x)\,d\nu(y)<\infty.
\]
Then the increasing quantile coupling is optimal:
\[
 \inf_{\pi\in\Pi(\sigma,\nu)}\int_{\R^2}\beta(|x-y|)\,d\pi
 =\int_0^1\beta(|F_\sigma^{-1}(t)-F_\nu^{-1}(t)|)\,dt.
\]
\end{thm}

\begin{lem}\label{lem:mass-shift}
For every $\theta\in[0,1]$, define
\[
 \nu_{\mathbf r,\theta}
 :=\theta\mu_{\mathbf r,-}+(1-\theta)\mu_{\mathbf r,+}.
\]
Then there exists $C_p>0$ such that
\[
 W_p(\nu_{\mathbf r,\theta},\mu_{\mathbf r})^p
 \le C_pA_{\mathbf r}|\theta-\lambda|^2,
\]
with $C_p$ independent of $r_-,r_+$ and $\theta$ for all sufficiently small
$r_-,r_+$.
\end{lem}

\begin{proof}
Let $F_-$ and $F_+$ be the distribution functions of
$\mu_{\mathbf r,-}$ and $\mu_{\mathbf r,+}$, and define the
\emph{inward-distance functions}
\[
 \psi_-(u):=a-F_-^{-1}(1-u),
 \qquad
 \psi_+(u):=F_+^{-1}(u)-a
 \qquad(0<u<1).
\]
For $\omega\in\{-,+\}$, define
\[
 \Psi_\omega(u):=
 \begin{cases}
 u/r_\omega^{\alpha_\omega},&0<u\le m_\omega,\\
 u^{1/(\alpha_\omega+1)},&m_\omega<u,
 \end{cases}
\]
and
\[
 w_\omega(u):=
 \begin{cases}
 r_\omega^{-\alpha_\omega},&0<u\le m_\omega,\\
 u^{-\alpha_\omega/(\alpha_\omega+1)},&m_\omega<u.
 \end{cases}
\]
We first establish estimates for these functions uniformly in small $\mathbf r$.
For $s>0$, put
\[
 I_-(s):=[a-s,a],
 \qquad
 I_+(s):=[a,a+s].
\]
Put $M_\omega(s):=\mu_{\mathbf r,\omega}(I_\omega(s))$ and
\[
 \Theta_\omega(s):=
 \begin{cases}
 r_\omega^{\alpha_\omega}s,&0<s\le r_\omega,\\
 s^{\alpha_\omega+1},&r_\omega<s\le s_0.
 \end{cases}
\]
The two branches of $\Theta_\omega$ agree at $r_\omega$, where their common
value is $m_\omega$.
By \eqref{eq:rhor}, there exist constants $c_0,C_0>0$, independent of small $\mathbf r$, such that for either sign $\omega\in\{-,+\}$ and all sufficiently small fixed $s_0<\delta$,
\begin{equation} \label{eq:Theta-M}
 c_0\Theta_\omega(s)\le M_\omega(s)\le C_0\Theta_\omega(s)
 \qquad(0<s\le s_0).
\end{equation}

We prove that there exist constants $u_0,c,C>0$, independent of small $\mathbf r$, such that
\begin{equation}\label{eq:psi}
 c\Psi_\omega(u)\le\psi_\omega(u)\le C\Psi_\omega(u)
 \qquad(0<u<u_0).
\end{equation}
Indeed, we have $\Theta_\omega^{-1}(u) = \Psi_\omega(u)$ for $0 < u \le s_0^{\alpha_\omega + 1}$.
The functions $M_\omega$ are absolutely continuous and strictly increasing,
and the definitions of $\psi_\omega$ give
$M_\omega(\psi_\omega(u))=u$. Choose $u_0>0$, independent of small
$\mathbf r$, so that
\[
 u_0<c_0\min_{\omega\in\{-,+\}}s_0^{\alpha_\omega+1}.
\]
\eqref{eq:Theta-M} implies that, for $0<u<u_0$,
\[
 \Theta_\omega^{-1}(u/C_0)
 \le\psi_\omega(u)
 \le\Theta_\omega^{-1}(u/c_0).
\]
For $b\in\{1/C_0,1/c_0\}$ and $0<u<u_0$, the choice of $u_0$ gives
$bu<s_0^{\alpha_\omega+1}$. Put $D:=\max\{b,b^{-1}\}$. If
$0<u\le m_\omega/D$, then
\[
 \Theta_\omega^{-1}(bu)
 =\frac{bu}{r_\omega^{\alpha_\omega}}
 =b\Psi_\omega(u).
\]
If $u\ge Dm_\omega$, then
\[
 \Theta_\omega^{-1}(bu)
 =b^{1/(\alpha_\omega+1)}\Psi_\omega(u).
\]
For $m_\omega/D\le u\le Dm_\omega$, both
$\Theta_\omega^{-1}(bu)$ and $\Psi_\omega(u)$ are comparable to
$r_\omega$. Consequently
$\Theta_\omega^{-1}(bu)\asymp\Psi_\omega(u)$, with constants depending
only on $b$ and the fixed data. This holds for both choices of $b$, and
together with the preceding inverse bounds gives \eqref{eq:psi}.

The lower bounds in \eqref{eq:rhor} and \eqref{eq:away-lower} imply that
$\rho_{\mathbf r}>0$ $\cL^1$-a.e.\ on each side of $a$.  Hence
\[
 M_-'(s)=\frac{\rho_{\mathbf r}(a-s)}{\lambda}>0,
 \qquad
 M_+'(s)=\frac{\rho_{\mathbf r}(a+s)}{1-\lambda}>0
\]
for $\cL^1$-a.e.\ $s$ in the respective domains.  Thus $M_\omega$ is
absolutely continuous with positive derivative a.e.  The inverse of such a
strictly increasing function is absolutely continuous, and the
inverse-function formula gives
\begin{equation}\label{eq:psi-derivative}
 \psi_-'(u)=\frac{\lambda}{\rho_{\mathbf r}(a-\psi_-(u))},
 \qquad
 \psi_+'(u)=\frac{1-\lambda}{\rho_{\mathbf r}(a+\psi_+(u))}.
\end{equation}
These identities hold for $\cL^1$-a.e.\ $u\in(0,1)$.

We now derive the estimate for these derivatives.  For $0<u<u_0$,
\eqref{eq:rhor} gives
\[
 \begin{aligned}
 \frac{1}{\rho_{\mathbf r}(a-\psi_-(u))}
 &\le C\min\left\{r_-^{-\alpha_-},\psi_-(u)^{-\alpha_-}\right\},\\
 \frac{1}{\rho_{\mathbf r}(a+\psi_+(u))}
 &\le C\min\left\{r_+^{-\alpha_+},\psi_+(u)^{-\alpha_+}\right\}.
 \end{aligned}
\]
If $0<u\le m_\omega$, the first term on the corresponding right-hand side is
$w_\omega(u)$.  If $m_\omega<u<u_0$, then \eqref{eq:psi} gives
\[
 \psi_\omega(u)\ge c\Psi_\omega(u)
 =c u^{1/(\alpha_\omega+1)},
\]
and hence the second term is at most
$Cu^{-\alpha_\omega/(\alpha_\omega+1)}=Cw_\omega(u)$.  Combining these
bounds with \eqref{eq:psi-derivative} yields
\begin{equation}\label{eq:psip}
 |\psi_\omega'(u)|\le Cw_\omega(u)
 \quad\text{for $\cL^1$-a.e.\ }u\in(0,u_0).
\end{equation}
Moreover, after decreasing the range of the sufficiently small parameters
$\mathbf r$ if necessary, we may assume that
$m_\omega<u_0/4$ for both signs. Put
\[
 d_0:=
 \min\left\{\delta,\,
 c\min_{\omega\in\{-,+\}}
 \left(\frac{u_0}{4}\right)^{1/(\alpha_\omega+1)}\right\}>0,
\]
where $c$ is the constant in \eqref{eq:psi}. Since $\psi_\omega$ is
increasing, \eqref{eq:psi} gives, for $u\ge u_0/4$,
\[
 \psi_\omega(u)\ge\psi_\omega(u_0/4)
 \ge c\Psi_\omega(u_0/4)\ge d_0.
\]
For sufficiently small $\mathbf r$, we also have $r_\omega<d_0$.
Thus the points $a-\psi_-(u)$ and $a+\psi_+(u)$ lie outside the
flattened intervals and stay at distance at least $d_0$ from $a$.
The lower bounds in \eqref{eq:power} and \eqref{eq:away-lower} therefore
give a constant $c_1>0$, independent of $\mathbf r$, such that
\[
 \rho_{\mathbf r}(a-\psi_-(u))\ge c_1,
 \qquad
 \rho_{\mathbf r}(a+\psi_+(u))\ge c_1
 \qquad(u\in[u_0/4,1)).
\]
Using \eqref{eq:psi-derivative}, we obtain
\begin{equation}\label{eq:psi-away}
 |\psi_\omega'(u)|\le c_1^{-1}
 \quad\text{for $\cL^1$-a.e.\ }u\in[u_0/4,1).
\end{equation}

Put $e:=|\theta-\lambda|$.  Choose
\[
 0<e_0<\frac12\min\{\lambda,1-\lambda\};
\]
we shall impose further upper bounds on $e_0$ below.  First assume
$0<e\le e_0$.  We write the monotone transport cost explicitly.  Let
$Q_\theta$ and $Q_\lambda$ denote the
generalized inverse distribution functions of $\nu_{\mathbf r,\theta}$ and
$\mu_{\mathbf r}$, respectively.  In terms of $\psi_-$ and $\psi_+$,
\[
 Q_\theta(t)=
 \begin{cases}
 a-\psi_-\left(1-\dfrac{t}{\theta}\right),&0<t<\theta,\\[2mm]
 a+\psi_+\left(\dfrac{t-\theta}{1-\theta}\right),&\theta<t<1,
 \end{cases}
\]
and the same formula with $\theta$ replaced by $\lambda$ gives
$Q_\lambda$.  By Theorem~\ref{thm:monotone},
\begin{equation}\label{eq:mass-quantile-cost}
 W_p(\nu_{\mathbf r,\theta},\mu_{\mathbf r})^p
 =\int_0^1|Q_\theta(t)-Q_\lambda(t)|^p\,dt.
\end{equation}

Suppose first that $\theta=\lambda+e$.  We split the integral in
\eqref{eq:mass-quantile-cost} into
\[
 (0,\lambda),\qquad(\lambda,\lambda+e),\qquad
 (\lambda+e,1).
\]
On $(0,\lambda)$ both $Q_\theta(t)$ and $Q_\lambda(t)$ lie to the left of
$a$.  Set $u:=1-t/\lambda$.  Then $dt=-\lambda\,du$, and $u$ decreases
from $1$ to $0$ as $t$ increases from $0$ to $\lambda$.  The inward parameter of
$Q_\lambda$ is $u$, whereas that of $Q_\theta$ is
\[
 1-\frac{t}{\lambda+e}
 =1-\frac{\lambda(1-u)}{\lambda+e}
 =u+\frac{(1-u)e}{\lambda+e}.
\]
Consequently, after reversing the integration limits, the contribution of
the interval $0<t<\lambda$ to \eqref{eq:mass-quantile-cost} is
\[
 \lambda\int_0^1
 \left|
 \psi_-\left(u+\frac{(1-u)e}{\lambda+e}\right)
 -\psi_-(u)
 \right|^p\,du.
\]
On $(\lambda+e,1)$ both $Q_\theta(t)$ and $Q_\lambda(t)$ lie to the right
of $a$.  Set
$v:=(t-\lambda-e)/(1-\lambda-e)$, so
$dt=(1-\lambda-e)\,dv$.  The inward parameter of $Q_\theta$ is $v$, while
that of $Q_\lambda$ is
\[
 \frac{t-\lambda}{1-\lambda}
 =v+\frac{(1-v)e}{1-\lambda}.
\]
Therefore, the contribution of this interval is
\[
 (1-\lambda-e)\int_0^1
 \left|
 \psi_+\left(v+\frac{(1-v)e}{1-\lambda}\right)
 -\psi_+(v)
 \right|^p\,dv.
\]
On the middle interval $(\lambda,\lambda+e)$, $Q_\theta(t)$ lies to the
left of $a$ and $Q_\lambda(t)$ lies to the right of $a$.  Its length is
exactly $e$, and the two inward parameters are
\[
 1-\frac{t}{\lambda+e}
 \le\frac{e}{\lambda+e},
 \qquad
 \frac{t-\lambda}{1-\lambda}
 \le\frac{e}{1-\lambda}.
\]
Hence both are at most $C_\lambda e$ for a constant $C_\lambda>0$
depending only on $\lambda$.

Suppose next that $\theta=\lambda-e$.  We now split the integral into
\[
 (0,\lambda-e),\qquad(\lambda-e,\lambda),\qquad(\lambda,1).
\]
On the first interval put $u:=1-t/(\lambda-e)$.  Then
$dt=-(\lambda-e)\,du$, and reversing the integration limits produces the
absolute Jacobian factor $\lambda-e$.  The inward parameter of
$Q_\theta$ is $u$, and that of $Q_\lambda$ is
\[
 1-\frac{t}{\lambda}
 =u+\frac{(1-u)e}{\lambda}.
\]
On the last interval put $v:=(t-\lambda)/(1-\lambda)$, so
$dt=(1-\lambda)\,dv$. The inward
parameter of $Q_\lambda$ is $v$, and that of $Q_\theta$ is
\[
 \frac{t-\lambda+e}{1-\lambda+e}
 =v+\frac{(1-v)e}{1-\lambda+e}.
\]
The crossing interval $(\lambda-e,\lambda)$ again has length $e$; there
the left inward parameter is at most $e/\lambda$ and the right one at most
$e/(1-\lambda+e)$.

After renaming $v$ as $u$ on the two right-hand intervals, each of the four
intervals on which $Q_\theta(t)$ and $Q_\lambda(t)$ lie on the same side of
$a$ contributes an integral of the form
\[
 J\int_0^1
 \left|\psi_\omega\bigl(u+h(u)\bigr)-\psi_\omega(u)\right|^p\,du,
\]
where the values of $\omega$, $J$, and $h$ are as follows:
\[
\begin{array}{c|c|c|c|c}
 \theta & \text{$t$-interval} & \omega & J & h(u)\\ \hline
 \lambda+e &(0,\lambda)&-&\lambda&
   \dfrac{(1-u)e}{\lambda+e}\\[2mm]
 \lambda+e &(\lambda+e,1)&+&1-\lambda-e&
   \dfrac{(1-u)e}{1-\lambda}\\[2mm]
 \lambda-e &(0,\lambda-e)&-&\lambda-e&
   \dfrac{(1-u)e}{\lambda}\\[2mm]
 \lambda-e &(\lambda,1)&+&1-\lambda&
   \dfrac{(1-u)e}{1-\lambda+e}.
\end{array}
\]
Here $J$ is the absolute Jacobian factor.  In particular, after enlarging
$C_\lambda$ if necessary, we have
\[
 0\le h(u)\le C_\lambda e
\]
in all four cases.  Each of the two remaining intervals, on which
$Q_\theta(t)$ and $Q_\lambda(t)$ lie on opposite sides of $a$, has length
$e$, and both inward parameters on it are at most $C_\lambda e$.  We may
assume that $C_\lambda\ge1$.  After decreasing $e_0$ so that
\[
 C_\lambda e_0<\frac{u_0}{4},
\]
each of the four factors $J$ is bounded above and below by positive
constants depending only on $\lambda$, and all inward parameters on the
two remaining intervals belong to $(0,u_0)$.

Choose $K_0>2C_\lambda$, put $K_1:=K_0+C_\lambda$, and decrease $e_0$
so that $K_1e_0<u_0/2$.  Suppose first that $0<u<K_0e$.  Since
$h(u)\ge0$ and $\psi_\omega$ is increasing,
\[
\begin{aligned}
 0
 &\le \psi_\omega(u+h(u))-\psi_\omega(u)\\
 &\le \psi_\omega(u+h(u))
 \le \psi_\omega(K_1e)
 \le C\Psi_\omega(K_1e),
\end{aligned}
\]
where the last inequality follows from \eqref{eq:psi}.  The factors $J$ are
uniformly bounded, and the interval $0<u<K_0e$ has length $K_0e$.
Consequently, its contribution is at most
$Ce\Psi_\omega(K_1e)^p$.

Suppose next that $K_0e\le u<u_0/2$.  Since
$h(u)\le C_\lambda e$ and $K_0>2C_\lambda$, we have
\[
 h(u)\le\frac{u}{2},
 \qquad
 u+h(u)\le\frac{3u}{2}<\frac{3u_0}{4}<u_0.
\]
The function $w_\omega$ is nonincreasing.  The absolute continuity of
$\psi_\omega$ and \eqref{eq:psip} therefore give
\[
\begin{aligned}
 |\psi_\omega(u+h(u))-\psi_\omega(u)|
 &\le \int_u^{u+h(u)}|\psi_\omega'(s)|\,ds
 \le C\int_u^{u+h(u)}w_\omega(s)\,ds\\
 &\le Ch(u)w_\omega(u)
 \le Ce\,w_\omega(u).
\end{aligned}
\]
If $u\ge u_0/2$, then both inward parameters belong to $(0,1)$ and
\eqref{eq:psi-away} gives
\[
 |\psi_\omega(u+h(u))-\psi_\omega(u)|
 \le \int_u^{u+h(u)}|\psi_\omega'(s)|\,ds
 \le c_1^{-1}h(u)
 \le Ce.
\]

On either interval where $Q_\theta(t)$ and $Q_\lambda(t)$ lie on opposite
sides of $a$, write $u_-$ and $u_+$ for their respective inward
parameters.  As shown above,
\[
 0<u_-,u_+\le C_\lambda e\le K_1e<u_0.
\]
The distance between the two points is therefore
\[
\begin{aligned}
 |Q_\theta(t)-Q_\lambda(t)|
 &=\psi_-(u_-)+\psi_+(u_+)\\
 &\le C\bigl(\Psi_-(u_-)+\Psi_+(u_+)\bigr)\\
 &\le C\bigl(\Psi_-(K_1e)+\Psi_+(K_1e)\bigr),
\end{aligned}
\]
where we used \eqref{eq:psi} and the monotonicity of $\Psi_-$ and
$\Psi_+$.  Since this interval has length $e$,
$(x+y)^p\le2^{p-1}(x^p+y^p)$ shows that its contribution is at most
\[
 C\sum_{\omega\in\{-,+\}}e\Psi_\omega(K_1e)^p.
\]

It remains to combine the three ranges considered above.  The portions
$0<u<K_0e$ of the two relevant intervals and the interval on which the
two points lie on opposite sides of $a$ contribute at most
\[
 C\sum_{\omega\in\{-,+\}}e\Psi_\omega(K_1e)^p.
\]
The portions $K_0e\le u<u_0/2$ contribute at most
\[
 C\sum_{\omega\in\{-,+\}}e^p
 \int_{K_0e}^{u_0/2}w_\omega(u)^p\,du,
\]
and the portions $u\ge u_0/2$ contribute at most $Ce^p$.  Put
$c_2:=2^{-1}\min\{1,K_0\}>0$.
Since $w_\omega^p$ is nonnegative,
\[
 \int_{K_0e}^{u_0/2}w_\omega(u)^p\,du
 \le\int_{c_2e}^{u_0}w_\omega(u)^p\,du.
\]
Combining these estimates yields
\begin{equation}\label{eq:mass-master}
 W_p(\nu_{\mathbf r,\theta},\mu_{\mathbf r})^p
 \le C\sum_{\omega\in\{-,+\}}
 \left[
 e\Psi_\omega(K_1e)^p
 +e^p\left(1+\int_{c_2e}^{u_0}w_\omega(u)^p\,du\right)
 \right].
\end{equation}

It remains to estimate each summand separately.  Fix one sign and abbreviate
$\alpha:=\alpha_\omega$, $r:=r_\omega$,
$m:=r^{\alpha+1}$, $\Psi:=\Psi_\omega$, $w:=w_\omega$, and
$A:=A_\omega$.  Suppose first that $0<e\le m$.  If $K_1e\le m$, then
$\Psi(K_1e)=K_1e r^{-\alpha}$.  If $K_1e>m$, then
\[
 \Psi(K_1e)=(K_1e)^{1/(\alpha+1)}
 \le K_1e r^{-\alpha},
\]
where the inequality follows from $K_1e>m=r^{\alpha+1}$.  Thus, in both
cases, $\Psi(K_1e)\le Ce r^{-\alpha}$.  Therefore
\[
 e\Psi(K_1e)^p
 \le Ce^{p+1}r^{-\alpha p}
 \le CAe^2,
\]
because $e^{p-1}r^{-\alpha p}
 \le m^{p-1}r^{-\alpha p}=r^{p-1-\alpha}\le A$.
Put $\gamma:=\alpha p/(\alpha+1)$.  Since $0<c_2<1$ and $e\le m$,
\[
 \int_{c_2e}^{u_0}w(u)^p\,du
 \le \int_{c_2e}^{m}r^{-\alpha p}\,du
     +\int_m^{u_0}u^{-\gamma}\,du
 \le r^{\alpha+1-\alpha p}
     +\int_m^{u_0}u^{-\gamma}\,du.
\]
If $\gamma<1$, the last integral is bounded; if $\gamma=1$, it is
$O(1+\log(1/r))$; and if $\gamma>1$, it is
$O(m^{1-\gamma})=O(r^{\alpha+1-\alpha p})$.  Since $p>2$,
\[
 e^{p-2}r^{\alpha+1-\alpha p}
 \le m^{p-2}r^{\alpha+1-\alpha p}
 =r^{p-1-\alpha}\le A.
\]
If $\gamma<1$, the bounded terms satisfy
$e^{p-2}\le1\le A$.  If $\gamma=1$, then
\[
 e^{p-2}\log(1/r)
 \le m^{p-2}\log(1/r),
\]
and the right-hand side is bounded because $p>2$.  Combining the estimates
obtained above in the three cases $\gamma<1$, $\gamma=1$, and $\gamma>1$,
we obtain
\[
 e^p\left(1+\int_{c_2e}^{u_0}w(u)^p\,du\right)\le CAe^2.
\]

Now suppose $m\le e\le e_0$.  Since $K_1>1$, we have $K_1e\ge m$ and
therefore
\[
 \Psi(K_1e)=(K_1e)^{1/(\alpha+1)}
 \le Ce^{1/(\alpha+1)}.
\]
Moreover, if $u\le m$, then
$w(u)=r^{-\alpha}\le u^{-\alpha/(\alpha+1)}$, while for $u>m$ we have
$w(u)=u^{-\alpha/(\alpha+1)}$ by definition.  Thus
$w(u)\le u^{-\alpha/(\alpha+1)}$ for every $u>0$.  Hence
\[
 e\Psi(K_1e)^p\le Ce^{1+p/(\alpha+1)}
\]
and
\[
 e^p\int_{c_2e}^{u_0}w(u)^p\,du
 \le Ce^p\int_{c_2e}^{u_0}u^{-\gamma}\,du,
 \qquad
 \gamma:=\frac{\alpha p}{\alpha+1}.
\]
If $\gamma<1$, the last integral is bounded, so this term is $O(e^p)$.
If $\gamma=1$, it is $O(e^p(1+\log(1/e)))$.  Since $p>2$ and
$0<e\le e_0<1$, both are $O(e^2)$.  If $\gamma>1$, then
\[
 e^p\int_{c_2e}^{u_0}u^{-\gamma}\,du
 \le Ce^{p+1-\gamma}
 =Ce^{1+p/(\alpha+1)}.
\]
Consequently the corresponding summand in \eqref{eq:mass-master} is at most
\[
 C\left(e^2+e^{1+p/(\alpha+1)}\right).
\]
If $\alpha\le p-1$, then $1+p/(\alpha+1)\ge2$ and $A=1$, so this is
$O(Ae^2)$.  If $\alpha>p-1$, put
$\beta:=p/(\alpha+1)-1<0$.  Since $e\ge m$,
\[
 e^{1+p/(\alpha+1)}
 =e^2e^\beta
 \le e^2m^\beta
 =r^{p-1-\alpha}e^2
 =Ae^2.
\]
Thus every summand in \eqref{eq:mass-master} is bounded by $CA_\omega e^2$,
and therefore
\[
 W_p(\nu_{\mathbf r,\theta},\mu_{\mathbf r})^p
 \le C(A_-+A_+)e^2
 \le CA_{\mathbf r}e^2
\]
for $0<e\le e_0$.

If $e\ge e_0$, both measures are supported on $[0,1]$, so
$W_p^p\le1\le e_0^{-2}A_{\mathbf r}e^2$.  The case $e=0$ is trivial.
This proves the lemma.
\end{proof}

\subsection{Regularized \texorpdfstring{$T_p$}{Tp} inequalities and upper bounds}

A probability measure $\sigma$ on an interval $I$ is said to satisfy a
\emph{Poincar\'e inequality} if there exists $C_{\mathrm P}<\infty$ such that
\[
 \Var_\sigma(f)
 \le C_{\mathrm P}\int_I |f'(x)|^2\,d\sigma(x)
\]
for every locally Lipschitz function $f:I\to\R$. The smallest possible value
of $C_{\mathrm P}$ is called the \emph{Poincar\'e constant} of $\sigma$, and
$1/C_{\mathrm{P}}$ is called the \emph{spectral gap}.

The following theorem is proved in \cite{Muckenhoupt}.
See also \cite{GozlanLine}*{(1.8)}.

\begin{thm}[Muckenhoupt criterion]\label{thm:Muck}
Let $d\sigma:=w(x)\,dx$ be a probability measure on an interval $I$, with $w>0$ $\cL^1$-a.e., and let $m$ be a median of $\sigma$. Put
\[
 B_+:=\sup_{x>m}\sigma([x,\sup I))\int_m^x\frac{dt}{w(t)},
\qquad B_-:=\sup_{x<m}\sigma((\inf I,x])\int_x^m\frac{dt}{w(t)}.
\]
Then $\sigma$ satisfies a Poincar\'e inequality iff $B_++B_-<\infty$, and the optimal Poincar\'e constant is comparable to $\max\{B_+,B_-\}$.\end{thm}

The following is due to \cite{GozlanLine}*{Theorem~1.1}.
\begin{thm}[Gozlan's one-dimensional criterion]\label{thm:Gozlan}
Let $\sigma$ be a Borel probability measure on $\R$, let
\[
 d\mu_1(x):=\frac12e^{-|x|}\,dx,
\]
and let $T:=F_\sigma^{-1}\circ F_{\mu_1}$. Let $h>0$, and let $\beta:[0,\infty)\to[0,\infty)$ be convex and satisfy $\beta(t)=t^2$ for $0\le t\le h$. We use the \emph{generalized inverse}
\[
 \beta^{-1}(s):=\inf\{t\ge0\mid\beta(t)\ge s\},\qquad s\ge0.
\]
For $b>0$, write $T_\beta(b)$ for the \emph{transport-entropy inequality}
\[
 \inf_{\pi\in\Pi(\nu,\sigma)}
 \int_{\R^2}\beta(b|x-y|)\,d\pi(x,y)
 \le H(\nu\mid\sigma)
 \qquad\text{for every Borel probability measure }\nu.
\]
Then the following are equivalent:
\begin{enumerate}[label=(\roman*),font=\upshape]
\item $T_\beta(b)$ holds for some $b>0$;
\item $\sigma$ satisfies a Poincar\'e inequality and there exists $d>0$ such
that
\[
 |T(x)-T(y)|\le d^{-1}\beta^{-1}(h^2+|x-y|)
 \qquad(x,y\in\R).
\]
\end{enumerate}
Moreover, for fixed $h$, the admissible transport constant $b$ can be
chosen quantitatively in terms of the Poincar\'e constant and the contraction
constant $d$.
\end{thm}

\begin{lem}\label{lem:reg-Tp}
For $p>2$, there exists $C_p>0$ such that $\mu_{\mathbf r}$ satisfies
$T_p(C_pA_{\mathbf r})$, with $C_p$ independent of $r_-,r_+$ for all
sufficiently small $r_-,r_+$, where $A_{\mathbf r}$ is defined in Section~\ref{ssec:Mass-shift}.
\end{lem}

\begin{proof}
We first prove uniform Poincar\'e inequalities for the two normalized side
measures.  Put
$\lambda_-:=\lambda$, $\lambda_+:=1-\lambda$,
$L_-:=a$, and $L_+:=1-a$, where $\lambda$ is defined in Section~\ref{ssec:Mass-shift}.  For $\omega\in\{-,+\}$, push
$\mu_{\mathbf r,\omega}$ forward by the inward coordinate $x\mapsto a-x$
when $\omega=-$ and $x\mapsto x-a$ when $\omega=+$, and denote the resulting
probability on $[0,L_\omega]$ by $\sigma_{\mathbf r,\omega}$.  Their densities are
\[
 \widehat w_{\mathbf r,-}(s)=\frac{\rho_{\mathbf r}(a-s)}{\lambda_-},
 \qquad
 \widehat w_{\mathbf r,+}(s)=\frac{\rho_{\mathbf r}(a+s)}{\lambda_+}.
\]
For either sign, \eqref{eq:rhor} gives, uniformly for small $\mathbf r$,
\begin{equation}\label{eq:side-mass-density}
 \begin{aligned}
 \sigma_{\mathbf r,\omega}([0,s])&\asymp
 \begin{cases}
 r_\omega^{\alpha_\omega}s,&0<s\le r_\omega,\\
 s^{\alpha_\omega+1},&r_\omega\le s<s_0,
 \end{cases}\\
 \widehat w_{\mathbf r,\omega}(s)&\asymp
 \begin{cases}
 r_\omega^{\alpha_\omega},&0<s\le r_\omega,\\
 s^{\alpha_\omega},&r_\omega\le s<s_0,
 \end{cases}
 \end{aligned}
\end{equation}
for any fixed sufficiently small $s_0<\delta$, with the density comparison
understood $\cL^1$-a.e.

Choose $s_0$ so small that
$\sigma_{\mathbf r,\omega}([0,s_0])<1/4$ for both signs and all sufficiently
small $\mathbf r$.  For each $\omega\in\{-,+\}$, let
$m_{\mathbf r,\omega}$ be a median of $\sigma_{\mathbf r,\omega}$. Then
$m_{\mathbf r,\omega}\ge s_0$.  Define the corresponding Muckenhoupt
quantities by
\[
 B_{\mathbf r,\omega}^+
 :=\sup_{m_{\mathbf r,\omega}<x<L_\omega}
 \sigma_{\mathbf r,\omega}([x,L_\omega])
 \int_{m_{\mathbf r,\omega}}^x
 \frac{ds}{\widehat w_{\mathbf r,\omega}(s)}
\]
and
\[
 B_{\mathbf r,\omega}^-
 :=\sup_{0<x<m_{\mathbf r,\omega}}
 \sigma_{\mathbf r,\omega}([0,x])
 \int_x^{m_{\mathbf r,\omega}}
 \frac{ds}{\widehat w_{\mathbf r,\omega}(s)}.
\]
By Theorem~\ref{thm:Muck}, it suffices to prove that
$B_{\mathbf r,\omega}^+$ and $B_{\mathbf r,\omega}^-$ are bounded uniformly
in sufficiently small $\mathbf r$ and in $\omega\in\{-,+\}$.
We first dispose of the parts that stay away from the bottleneck.  By
\eqref{eq:power} on $[s_0/2,\delta]$ and by \eqref{eq:away-lower} beyond
$\delta$, there is a constant $c_{s_0}>0$, independent of small
$\mathbf r$, such that
\begin{equation}\label{eq:side-away-density-lower}
 \widehat w_{\mathbf r,\omega}(s)\ge c_{s_0}
 \quad\text{for $\cL^1$-a.e.\ }s\in[s_0/2,L_\omega].
\end{equation}
Hence, if $s_0/2\le x<m_{\mathbf r,\omega}$, since $m_{\mathbf r,\omega} \le L_\omega$,
\[
 \sigma_{\mathbf r,\omega}([0,x])
 \int_x^{m_{\mathbf r,\omega}}
 \frac{ds}{\widehat w_{\mathbf r,\omega}(s)}
 \le \frac{L_\omega}{c_{s_0}},
\]
and, for every $x>m_{\mathbf r,\omega}$,
\[
 \sigma_{\mathbf r,\omega}([x,L_\omega])
 \int_{m_{\mathbf r,\omega}}^x
 \frac{ds}{\widehat w_{\mathbf r,\omega}(s)}
 \le \frac{L_\omega}{c_{s_0}}.
\]
Thus $B_{\mathbf r,\omega}^+$ is uniformly bounded, and the only part of
$B_{\mathbf r,\omega}^-$ that remains to be estimated is the range
$0<x<s_0/2$.  For $0<x\le r_\omega$,
\eqref{eq:side-mass-density} and \eqref{eq:side-away-density-lower} give
\begin{equation}\label{eq:Bminus-small-x}
 \sigma_{\mathbf r,\omega}([0,x])
 \int_x^{m_{\mathbf r,\omega}}
 \frac{ds}{\widehat w_{\mathbf r,\omega}(s)}
 \le
 Cr_\omega^{\alpha_\omega}x
 \left[
 1+\int_x^{r_\omega}r_\omega^{-\alpha_\omega}\,ds
 +\int_{r_\omega}^{s_0}s^{-\alpha_\omega}\,ds
 \right].
\end{equation}
If $r_\omega\le x<s_0/2$, \eqref{eq:side-mass-density} and
\eqref{eq:side-away-density-lower} give
\begin{equation}\label{eq:Bminus-intermediate-x}
 \sigma_{\mathbf r,\omega}([0,x])
 \int_x^{m_{\mathbf r,\omega}}
 \frac{ds}{\widehat w_{\mathbf r,\omega}(s)}
 \le
 Cx^{\alpha_\omega+1}
 \left[1+\int_x^{s_0}s^{-\alpha_\omega}\,ds\right].
\end{equation}
For any $\alpha>0$,
\[
 \int_u^{s_0}s^{-\alpha}\,ds
 =\begin{cases}
 \dfrac{s_0^{1-\alpha}-u^{1-\alpha}}{1-\alpha},&\alpha<1,\\[2mm]
 \log(s_0/u),&\alpha=1,\\[2mm]
 \dfrac{u^{1-\alpha}-s_0^{1-\alpha}}{\alpha-1},&\alpha>1.
 \end{cases}
\]
We now prove that the right-hand sides of \eqref{eq:Bminus-small-x} and
\eqref{eq:Bminus-intermediate-x} are uniformly bounded. If
$0<x\le r_\omega$ and
$\alpha_\omega<1$, the integral
$\int_{r_\omega}^{s_0}s^{-\alpha_\omega}\,ds$ is bounded independently of
$r_\omega$, while
$r_\omega^{\alpha_\omega}x\int_x^{r_\omega}r_\omega^{-\alpha_\omega}\,ds
\le xr_\omega$; hence the right-hand side of
\eqref{eq:Bminus-small-x} is bounded. If
$\alpha_\omega=1$, the integral
$\int_{r_\omega}^{s_0}s^{-\alpha_\omega}\,ds$ equals
$\log(s_0/r_\omega)$, so the term
$Cr_\omega^{\alpha_\omega}x
\int_{r_\omega}^{s_0}s^{-\alpha_\omega}\,ds$ on the right-hand side of
\eqref{eq:Bminus-small-x} is at most
$Cr_\omega x\log(s_0/r_\omega)$, which is uniformly bounded. If
$\alpha_\omega>1$, the integral
$\int_{r_\omega}^{s_0}s^{-\alpha_\omega}\,ds$ is at most
$Cr_\omega^{1-\alpha_\omega}$, and hence the term
$Cr_\omega^{\alpha_\omega}x
\int_{r_\omega}^{s_0}s^{-\alpha_\omega}\,ds$ on the right-hand side of
\eqref{eq:Bminus-small-x} is at most
$Cr_\omega x\le Cr_\omega^2$. Thus the part of
$B_{\mathbf r,\omega}^-$ corresponding to $0<x\le r_\omega$ is uniformly
bounded in all cases.

Now let $r_\omega\le x<s_0/2$. If $\alpha_\omega<1$, the integral
$\int_x^{s_0}s^{-\alpha_\omega}\,ds$ is bounded, so the right-hand side of
\eqref{eq:Bminus-intermediate-x} is $O(x^{\alpha_\omega+1})$. If
$\alpha_\omega=1$, the right-hand side of
\eqref{eq:Bminus-intermediate-x} is
$O(x^2(1+\log(s_0/x)))$, which is bounded on $(0,s_0/2)$. If
$\alpha_\omega>1$, the integral
$\int_x^{s_0}s^{-\alpha_\omega}\,ds$ is
$O(x^{1-\alpha_\omega})$, and the right-hand side of
\eqref{eq:Bminus-intermediate-x} is therefore $O(x^2)$. Hence
the part of $B_{\mathbf r,\omega}^-$ corresponding to
$r_\omega\le x<s_0/2$ is also uniformly bounded.
Theorem~\ref{thm:Muck} gives Poincar\'e constants for
$\sigma_{\mathbf r,-}$ and $\sigma_{\mathbf r,+}$ that are independent of
$\mathbf r$. Since the inward-coordinate maps are isometries and
$\sigma_{\mathbf r,\omega}$ is the push-forward of
$\mu_{\mathbf r,\omega}$, the same uniform Poincar\'e constants apply to
$\mu_{\mathbf r,-}$ and $\mu_{\mathbf r,+}$.

The supports of $\mu_{\mathbf r,-}$ and $\mu_{\mathbf r,+}$ have diameter at
most $1$.  Let $T_{\mathbf r,\omega}:\R\to\R$ be the monotone transport
from the probability measure $\mu_1$ defined in
Theorem~\ref{thm:Gozlan} to $\mu_{\mathbf r,\omega}$.  The image of
$T_{\mathbf r,\omega}$ is contained in
$\operatorname{supp}\mu_{\mathbf r,\omega}$, so
\[
 |T_{\mathbf r,\omega}(x)-T_{\mathbf r,\omega}(y)|\le1
 \qquad(x,y\in\R).
\]
Apply Theorem~\ref{thm:Gozlan} with $\beta(t)=t^2$ and a fixed $h>0$.
Then $\beta^{-1}(h^2+|x-y|)=\sqrt{h^2+|x-y|}\ge h$ and
\[
 |T_{\mathbf r,\omega}(x)-T_{\mathbf r,\omega}(y)|
 \le1\le h^{-1}\sqrt{h^2+|x-y|}.
\]
Together with the uniform Poincar\'e inequalities proved above, the
implication (ii) $\Rightarrow$ (i) and the quantitative
conclusion of Theorem~\ref{thm:Gozlan} now give a constant $a_0>0$, independent of small $\mathbf r$,
 such that $T_\beta(a_0)$
holds for both $\mu_{\mathbf r,-}$ and $\mu_{\mathbf r,+}$. Since
$\beta(t)=t^2$, the definition of $T_\beta(a_0)$ gives, for every Borel
probability measure $\eta$ on $\R$,
\[
 W_2(\eta,\mu_{\mathbf r,\omega})^2
 \le C H(\eta\mid\mu_{\mathbf r,\omega}).
\]
Here $C:=a_0^{-2}$.
If $H(\eta\mid\mu_{\mathbf r,\omega})=+\infty$, the inequality
$W_p(\eta,\mu_{\mathbf r,\omega})^p\le
C H(\eta\mid\mu_{\mathbf r,\omega})$ is automatic. If
$H(\eta\mid\mu_{\mathbf r,\omega})<+\infty$, then
$\eta\ll\mu_{\mathbf r,\omega}$, so both $\eta$ and
$\mu_{\mathbf r,\omega}$ are supported on
$\operatorname{supp}\mu_{\mathbf r,\omega}$, whose diameter is at most $1$.
Every $\pi\in\Pi(\eta,\mu_{\mathbf r,\omega})$ is therefore supported on
pairs with $|x-y|\le1$, and since $p>2$, $|x-y|^p\le|x-y|^2$. Taking the
infimum over $\pi\in\Pi(\eta,\mu_{\mathbf r,\omega})$ gives
\begin{equation}\label{eq:side-Tp}
 W_p(\eta,\mu_{\mathbf r,\omega})^p
 \le W_2(\eta,\mu_{\mathbf r,\omega})^2
 \le C H(\eta\mid\mu_{\mathbf r,\omega}).
\end{equation}

Let now $\nu$ be an arbitrary probability measure on $[0,1]$.  If
$\nu\not\ll\mu_{\mathbf r}$, then
$H(\nu\mid\mu_{\mathbf r})=+\infty$, so the inequality
$W_p(\nu,\mu_{\mathbf r})^p\le
CA_{\mathbf r}H(\nu\mid\mu_{\mathbf r})$ is automatic. Assume therefore that
$\nu\ll\mu_{\mathbf r}$ and put $\theta:=\nu([0,a])$.
Since $\nu\ll\mu_{\mathbf r}\ll\cL^1$, we have $\nu(\{a\})=0$.  Thus
$\theta=\nu([0,a))$ and the restrictions to $[0,a]$ and $[a,1]$ do not
overlap in positive $\nu$-measure.
If $\theta>0$, define $\nu_-:=\theta^{-1}\nu|_{[0,a]}$,
and if $\theta<1$, define $\nu_+:= (1-\theta)^{-1}\nu|_{[a,1]}$.
When $\theta=0$, set $\nu_-:=\mu_{\mathbf r,-}$, and when $\theta=1$, set
$\nu_+:=\mu_{\mathbf r,+}$. Set
\[
 \nu^*:=\theta\mu_{\mathbf r,-}+(1-\theta)\mu_{\mathbf r,+}.
\]
Fix $\varepsilon>0$.  Choose couplings $\pi_\omega$ between $\nu_\omega$ and $\mu_{\mathbf r,\omega}$, $\omega \in \{-,+\}$, such that
\[
 \int_{[0,a]^2}|x-y|^p\,d\pi_-
 \le W_p(\nu_-,\mu_{\mathbf r,-})^p+\varepsilon,
\]
and
\[
 \int_{[a,1]^2}|x-y|^p\,d\pi_+
 \le W_p(\nu_+,\mu_{\mathbf r,+})^p+\varepsilon.
\]
The convex combination $\theta\pi_-+(1-\theta)\pi_+$ is a coupling of $\nu$ and $\nu^*$, and after
letting $\varepsilon\downarrow0$ we obtain
\[
 W_p(\nu,\nu^*)^p
 \le\theta W_p(\nu_-,\mu_{\mathbf r,-})^p
 +(1-\theta)W_p(\nu_+,\mu_{\mathbf r,+})^p.
\]
Using \eqref{eq:side-Tp} gives
\begin{equation}\label{eq:side-transport-entropy}
 W_p(\nu,\nu^*)^p
 \le C\left[
 \theta H(\nu_-\mid\mu_{\mathbf r,-})
 +(1-\theta)H(\nu_+\mid\mu_{\mathbf r,+})
 \right].
\end{equation}
Define
\[
 D_{\mathrm{Ber}}(\theta\Vert\lambda)
 :=\theta\log\frac\theta\lambda
 +(1-\theta)\log\frac{1-\theta}{1-\lambda},
\]
with endpoint values defined by continuity.  There exists $c_\lambda>0$ such
that
\begin{equation}\label{eq:Ber-quadratic}
 D_{\mathrm{Ber}}(\theta\Vert\lambda)
 \ge c_\lambda|\theta-\lambda|^2
 \qquad(0\le\theta\le1).
\end{equation}
Indeed,
\[
 \frac{d^2}{d\theta^2}D_{\mathrm{Ber}}(\theta\Vert\lambda)
 =\frac1{\theta(1-\theta)}\ge4
 \qquad(0<\theta<1).
\]
Since
$D_{\mathrm{Ber}}(\lambda\Vert\lambda)
=D_{\mathrm{Ber}}'(\lambda\Vert\lambda)=0$, Taylor's theorem at
$\theta=\lambda$ gives
$D_{\mathrm{Ber}}(\theta\Vert\lambda)\ge2|\theta-\lambda|^2$ for
$0<\theta<1$. The same inequality at $\theta=0,1$ follows by continuity,
and hence \eqref{eq:Ber-quadratic} holds with $c_\lambda=2$.

Lemma~\ref{lem:mass-shift} and \eqref{eq:Ber-quadratic} therefore give
\begin{equation}\label{eq:mass-entropy}
 W_p(\nu^*,\mu_{\mathbf r})^p
 \le CA_{\mathbf r}D_{\mathrm{Ber}}(\theta\Vert\lambda).
\end{equation}
We compute the entropy decomposition directly.  On $[0,a]$, when
$\theta>0$,
\[
 \frac{d\nu}{d\mu_{\mathbf r}}
 =\frac\theta\lambda
 \frac{d\nu_-}{d\mu_{\mathbf r,-}}
 \qquad \mu_{\mathbf r}\text{-a.e.}
\]
Therefore
\begin{equation}\label{eq:entropy-left}
 \int_{[0,a]}\log\frac{d\nu}{d\mu_{\mathbf r}}\,d\nu
 =\theta\log\frac\theta\lambda
 +\theta H(\nu_-\mid\mu_{\mathbf r,-}).
\end{equation}
If $\theta=0$, both $\theta\log(\theta/\lambda)$ and
$\theta H(\nu_-\mid\mu_{\mathbf r,-})$ are interpreted as $0$. On $[a,1]$,
when $1-\theta>0$,
\[
 \frac{d\nu}{d\mu_{\mathbf r}}
 =\frac{1-\theta}{1-\lambda}
 \frac{d\nu_+}{d\mu_{\mathbf r,+}}
 \qquad \mu_{\mathbf r}\text{-a.e.},
\]
and hence
\begin{equation}\label{eq:entropy-right}
 \int_{[a,1]}\log\frac{d\nu}{d\mu_{\mathbf r}}\,d\nu
 =(1-\theta)\log\frac{1-\theta}{1-\lambda}
 +(1-\theta)H(\nu_+\mid\mu_{\mathbf r,+}).
\end{equation}
At $\theta=1$, both
$(1-\theta)\log((1-\theta)/(1-\lambda))$ and
$(1-\theta)H(\nu_+\mid\mu_{\mathbf r,+})$ are interpreted as $0$. Adding
\eqref{eq:entropy-left} and \eqref{eq:entropy-right} gives
\begin{equation}\label{eq:entropy-decomp}
 H(\nu\mid\mu_{\mathbf r})
 =D_{\mathrm{Ber}}(\theta\Vert\lambda)
 +\theta H(\nu_-\mid\mu_{\mathbf r,-})
 +(1-\theta)H(\nu_+\mid\mu_{\mathbf r,+}).
\end{equation}

Finally the triangle inequality and $(x+y)^p\le2^{p-1}(x^p+y^p)$ give,
using \eqref{eq:side-transport-entropy}, \eqref{eq:mass-entropy},
$A_{\mathbf r}\ge1$, and \eqref{eq:entropy-decomp},
\[
 W_p(\nu,\mu_{\mathbf r})^p
 \le CA_{\mathbf r}H(\nu\mid\mu_{\mathbf r}).
\]
Thus $\mu_{\mathbf r}$ satisfies $T_p(CA_{\mathbf r})$.
\end{proof}

\begin{prop}\label{prop:upper-gt2}
For $p>2$ and every fixed $0<\kappa<1$, as $n\to\infty$,
\[
 \ObsDiam(X_{n,p};-\kappa)
 \lesssim
 \begin{cases}
 1,&\alpha_*\le p-1,\\
 n^{\frac1{2p}-\frac1{2(\alpha_*+1)}},&\alpha_*>p-1.
 \end{cases}
\]
\end{prop}

\begin{proof}
Fix $\eta\in(0,\kappa)$.
Assume first that $\alpha_*\le p-1$. Then
$A_{\mathbf r}=1$ for every sufficiently small $r_-,r_+$. By
Lemma~\ref{lem:reg-Tp} and Proposition~\ref{prop:tensor-Tp},
$\mu_{\mathbf r}^{\otimes n}$ satisfies $T_p(C)$ on
$([0,1]^n,d_p)$ with $C$ independent of $n$ and $\mathbf r$.
Lemma~\ref{lem:Tp-obs} gives
\[
 \ObsDiam([0,1]^n,d_p,\mu_{\mathbf r}^{\otimes n};-(\kappa-\eta))
 \le C_{p,\kappa,\eta}.
\]
Lemma~\ref{lem:Wp-comp} and \eqref{eq:Wp-reg} therefore give
\[
 \ObsDiam(X_{n,p};-\kappa)
 \le C_{p,\kappa,\eta}
 +Cn^{1/p}\left(r_-^{\alpha_-+p+1}+r_+^{\alpha_++p+1}\right)^{1/p}.
\]
For each fixed $n$, letting $r_-,r_+\downarrow0$ gives
$\ObsDiam(X_{n,p};-\kappa)\le C_{p,\kappa,\eta}$, where
$C_{p,\kappa,\eta}$ is independent of $n$.

Now assume $\alpha_*>p-1$. For all sufficiently large $n$, choose
$r_{\pm,n}:=n^{-1/(2(\alpha_\pm+1))}$ and
$\mathbf r_n:=(r_{-,n},r_{+,n})$. Lemma~\ref{lem:reg-Tp},
Proposition~\ref{prop:tensor-Tp}, Lemma~\ref{lem:Tp-obs},
Lemma~\ref{lem:Wp-comp}, and \eqref{eq:Wp-reg} give
\begin{equation}\label{eq:upper-gt2-balance}
 \ObsDiam(X_{n,p};-\kappa)
 \le C\left[
 A_{\mathbf r_n}^{1/p}
 +n^{1/p}\left(r_{-,n}^{\alpha_-+p+1}
 +r_{+,n}^{\alpha_++p+1}\right)^{1/p}
 \right].
\end{equation}
For $\omega\in\{-,+\}$ put
$\beta_\omega:=1/(2p)-1/(2(\alpha_\omega+1))$ and
$\beta_*:=1/(2p)-1/(2(\alpha_*+1))$. Since the function
$\alpha\mapsto1/(2p)-1/(2(\alpha+1))$ is increasing,
$\beta_\omega\le\beta_*$. If
$\alpha_\omega>p-1$, then
$r_{\omega,n}^{p-1-\alpha_\omega}=n^{p\beta_\omega}$,
while for $\alpha_\omega\le p-1$,
$r_{\omega,n}^{p-1-\alpha_\omega}\le1$. Hence
$A_{\mathbf r_n}^{1/p}\le Cn^{\beta_*}$. Moreover,
\[
 n^{1/p}r_{\omega,n}^{(\alpha_\omega+p+1)/p}=n^{\beta_\omega}.
\]
Using $(u+v)^{1/p}\le u^{1/p}+v^{1/p}$, the second term inside the square
brackets on the right-hand side of \eqref{eq:upper-gt2-balance} is at most
$C(n^{\beta_-}+n^{\beta_+})\le Cn^{\beta_*}$. Together with
$A_{\mathbf r_n}^{1/p}\le Cn^{\beta_*}$, equation
\eqref{eq:upper-gt2-balance} gives
\[
 \ObsDiam(X_{n,p};-\kappa)
 \le Cn^{\frac1{2p}-\frac1{2(\alpha_*+1)}}.
\]
\end{proof}

\begin{prop}\label{prop:lower-gt2}
For $p>2$ and every fixed $0<\kappa<1$, as $n\to\infty$,
\[
 \ObsDiam(X_{n,p};-\kappa)
 \gtrsim
 \begin{cases}
 1,&\alpha_*\le p-1,\\
 n^{\frac1{2p}-\frac1{2(\alpha_*+1)}},&\alpha_*>p-1.
 \end{cases}
\]
\end{prop}

\begin{proof}
Fix $0<\kappa<1$. Suppose first that $\alpha_*\le p-1$. The
first-coordinate projection $P_1:X_{n,p}\to[0,1]$ is $1$-Lipschitz and
pushes $\mu^{\otimes n}$ forward to $\mu$. Hence
Lemma~\ref{lem:positive-partdiam} gives
\[
 \ObsDiam(X_{n,p};-\kappa)\ge\PartDiam(\mu;1-\kappa)>0,
\]
which proves $\ObsDiam(X_{n,p};-\kappa)\gtrsim1$ when
$\alpha_*\le p-1$.

Now suppose $\alpha_*>p-1$. Then
$\beta_{p,*}:=1/(2p)-1/(2(\alpha_*+1))>0$. Fix
$0<\kappa_0<1/2$. Lemma~\ref{lem:majority} gives, for all sufficiently large
$m$, $\ObsDiam(X_{m,p};-\kappa_0)\ge c m^{\beta_{p,*}}$. Applying
Lemma~\ref{lem:block} with target level $\kappa$, we obtain a fixed integer
$k=k(\kappa,\kappa_0)$ and $c'>0$ such that
$\ObsDiam(X_{km,p};-\kappa)\ge c'm^{\beta_{p,*}}$. For a general
sufficiently large $n$, put $m:=\lfloor n/k\rfloor$ and let
$\pi_{km}:X_{n,p}\to X_{km,p}$ be the projection onto the first $km$
coordinates. The map $\pi_{km}$ is $1$-Lipschitz and satisfies
$(\pi_{km})_*\mu^{\otimes n}=\mu^{\otimes km}$, so
\[
 \ObsDiam(X_{n,p};-\kappa)
 \ge\ObsDiam(X_{km,p};-\kappa)
 \ge c'm^{\beta_{p,*}}.
\]
Since $k$ is fixed, $m\asymp n$, and hence
\[
 \ObsDiam(X_{n,p};-\kappa)
 \ge c n^{\frac1{2p}-\frac1{2(\alpha_*+1)}}.
\]
\end{proof}

\section{The case \texorpdfstring{$p=1$}{p=1}}

\begin{prop}\label{prop:p1}
For every fixed $0<\kappa<1$, as $n\to\infty$,
\[
 \ObsDiam(X_{n,1};-\kappa)\asymp\sqrt n.
\]
\end{prop}

\begin{proof}
The mm-space $([0,1],|\cdot|,\mu)$ is nontrivial and has finite diameter.
Therefore the upper bound follows from
\cite{OzawaShioyaProduct}*{Theorem~1.1}, while the lower bound follows from
\cite{OzawaShioyaProduct}*{Theorem~1.3}. Consequently,
\[
 \ObsDiam(X_{n,1};-\kappa)\asymp\sqrt n.
\]
\end{proof}

\section{Proof of the main theorem}

\begin{proof}[Proof of Main Theorem]
For $p=1$, use Proposition~\ref{prop:p1}. For $1<p\le2$, combine
Propositions~\ref{prop:upper-le2} and~\ref{prop:lower-le2}. For
$2<p<\infty$, combine Propositions~\ref{prop:upper-gt2} and
\ref{prop:lower-gt2}.  Finally let $p=\infty$. Since
$\diam([0,1]^n,d_\infty)=1$, we have
$\ObsDiam(X_{n,\infty};-\kappa)\le1$.  On the other hand, the first-coordinate
projection is $1$-Lipschitz for $d_\infty$ and pushes $\mu^{\otimes n}$ forward
to $\mu$.  Lemma~\ref{lem:positive-partdiam} gives
$\PartDiam(\mu;1-\kappa)>0$, and therefore
$\ObsDiam(X_{n,\infty};-\kappa)\ge\PartDiam(\mu;1-\kappa)>0$.
This proves Main Theorem for $p=\infty$.
\end{proof}

\begin{ack}
The author used ChatGPT in preparing this paper. The author takes full responsibility for its content.
\end{ack}

\end{document}